\documentclass[notitlepage,12pt]{article}

\usepackage[top=0.5in,bottom=0.7in]{geometry}

\usepackage{graphicx}
\counterwithin{figure}{section}
\counterwithin{table}{section}
\usepackage{bm}
\usepackage{xcolor}
\usepackage{imakeidx}

\usepackage{tikz}
\usetikzlibrary{shapes,patterns,patterns.meta,calc,decorations,positioning}

\usepackage[bookmarks=false]{hyperref}

\usepackage[style=alphabetic]{biblatex}
\usepackage{amsmath, amsthm, amssymb}
\usepackage{mathtools}
\usepackage[justification=raggedright]{caption}
\usepackage{xspace}
\usepackage{complexity}
\usepackage{bbold}
\usepackage{enumitem}
\usepackage{physics}

\usepackage{cleveref}

\def\allowUnusedRef#1{}

\usepackage{derivative}
\usepackage{mathrsfs}

\usepackage{fancyhdr}
\makeindex[title={Index\label{index}},columns=2,intoc]

\usepackage{float}

\definecolor{orange(colorwheel)}{rgb}{1.0, 0.5, 0.0}

\newcommand{\df}[1]{\textbf{#1}\index{#1}}
\newcommand{\mathdf}[1]{#1\indexmath{#1}}
\def\indexmath#1{\index{\ensuremath{#1}}}

\newcommand{\CR}{\color{red}}

\numberwithin{equation}{subsection}
\newtheorem{theorem}{Theorem}[subsection]

\newtheorem{lemma}[theorem]{Lemma}

\newtheorem{proposition}[theorem]{Proposition}

\theoremstyle{definition}

\theoremstyle{definition}

\theoremstyle{remark}

\newtheorem*{rep@theorem}{\rep@title}
\newcommand{\newreptheorem}[2]{%
\newenvironment{rep#1}[1]{%
 \def\rep@title{#2 \ref{##1}}%
 \begin{rep@theorem}}%
 {\end{rep@theorem}}}
\makeatother

\newreptheorem{thm}{Theorem}
\newreptheorem{lem}{Lemma}

\def\fin{\breve}

\def\cTurnK{\tilde c_\ourK}
\def\muTurnK{\tilde \mu_\ourK}
\def\muK{\mu_\ourK}

\newcommand{\floor}[1]{\lfloor #1 \rfloor}

\newcommand{\ov}[1]{\overline{#1}}

\def\ovalpha{{1 -\alpha}}

\title{The Satisfiability Threshold for K-XOR Games}

\author{Jared~A.~Hughes$^*$ and J.~William~Helton\footnote{University of California, San Diego}} 
\date{\today}

\def\R{{\mathbb R}}
\def\bbZ{{\mathbb Z}}

\def\ben{\begin{enumerate}}
\def\een{\end{enumerate}}

\def\z2{{\mathbb Z}_2}
 
\def\rkz2{{rk_{\z2}}}

\newcommand{\half}{\frac{1}{2}}

\def\expMTwo{\exp_2}
\newcommand\cA{\mathcal{A}}

\def\paper{paper}

\DeclareMathOperator{\Ex}{\mathbb{E}}

\let\Pr\undefined
\DeclareMathOperator{\Pr}{Pr}

\def\zzz{\boldsymbol{\zeta}}
\def\zzzlin{\zzz_{lin}}
\def\zzzsqrt{\zzz_{sqrt}}

\def\zzzComp{\widehat{\zzz}_{\ourK,c}}
\def\zzzKc{\zzz_{\ourK,c}}
\def\zzzKci{\zzz_{\ourK,c_i}}
\def\zzzHelper{\zzz^*_{\ourK}}
\def\zetaHelper{\zeta^*}
\def\deltaBallK{{\hat\delta_\ourK}}
\def\deltaK{\delta_\ourK}
\def\epsBallK{{\hat\epsilon_\ourK}}
\def\epsK{\epsilon_\ourK}
\def\zetaLowerK{\underline{\zeta}_\ourK}
\def\betaK{\beta_\ourK}
\def\zetaLLK{\eta_\ourK}
\def\epsLLK{\hat\eta_\ourK}

\def\aas{a.a.s.\@}

\def\bs{\bigskip}

\def\littlekIntro{K}
\def\ourK{K}
\def\littlek{k}
\def\Hk{H_\littlek}

\def\Hone{H_1}

\def\JK{J_{\ourK}}

\def\LK{L_{\ourK}}

\def\bn{{\bm{n}}}

\def\Ximn{\Xi_{m,\bn}}

\def\Psimn{\Psi_{m,\bn}}
\def\Psimnfin{\Psi_{\fin m,\fin \bn}}

\def\kXORSAT{{$\littlek$-XORSAT}}
\def\kXORSATintro{{$\littlekIntro$-XORSAT}}
\def\KXORGAME{{$\ourK$-XORGAME}}

\def\sigmn{{\abs{\bn}}}
\def\As{{(A_1, A_2, \cdots, A_\ourK)}}

\def\cThreshK{{c^*_\ourK}}

\def\Zmnl{Z_{m,\bn}^{(\ell)}}
\def\Ymnkl{Y_{m,n,\littlek}^{(\ell)}}

\def\mypdv#1#2{{\frac{\partial#1}{\partial#2}}}

\def\textforall{\qquad\mathrm{for\ all}\ }
\def\foralldisplay#1{\textforall #1}

\def\alphaloK{\underline{\alpha}_\ourK}
\def\alphahiK{\overline{\alpha}_\ourK}

\begin{document}

\allowUnusedRef{index}
\allowUnusedRef{table-of-contents}
\allowUnusedRef{eq:nb-1-LK-neg}
\allowUnusedRef{fig:nb-Jklin456}
\allowUnusedRef{eq:nb-upper-bound-LK}
\allowUnusedRef{lem:nb-Kgeq4-J-bound}
\allowUnusedRef{lem:nb-K3-Jsqrt-bound}
\allowUnusedRef{eq:nb-2a-Jsqrt3-regime}
\allowUnusedRef{eq:nb-2b-Jsqrt3-regime}
\allowUnusedRef{eq:nb-4-Jsqrt3-at-half}
\allowUnusedRef{eq:nb-1a-beta-alpha-fits}
\allowUnusedRef{tab:nb-computed-cThreshK}
\allowUnusedRef{fig:nb-JKLin_K4_varying_c}
\allowUnusedRef{sec:nb-bounds-along-zzzlin}
\allowUnusedRef{eq:nb-3-Jsqrt3-second-partial-alpha}

\date{\today}
\maketitle
\begin{abstract}

A $\ourK$-XORGAME system corresponds to a $\ourK$-XORSAT system with the additional restriction that the variables divide uniformly into $\ourK$ blocks. This forms a system of $m$ equations with $\ourK n$ unknowns over $\mathbb Z_2$, and a perfect strategy corresponds to a solution to these equations. Equivalently, such equations correspond to colorings of a $\ourK$-uniform $\ourK$-partite hypergraph. 
This paper  proves that the satisfiability 
threshold of $m/n$ for \KXORGAME{} problems exists
and equals the satisfiability threshold for \kXORSATintro{}.

\end{abstract}

\section{Introduction}

This paper concerns the satisfiability threshold for existence of a perfect strategy 
for a 
\KXORGAME, effectively a linear algebra problem over the binaries. 
A similar-looking structure 
is found in the classical \kXORSATintro{} problem.
There has been heavy study of the satisfiability threshold for the \kXORSATintro{}
problem,
with the celebrated  paper of Dubois-Mandler \cite{dBM-FOCS} determining the threshold for $\littlekIntro=3$, and subsequent results (\cite{pittel2016satisfiability}) determining the threshold for \kXORSATintro{} with general $\littlekIntro$,
giving the exact value of the satisfiability threshold as the solution to an explicit transcendental equation.

There is much study of XORSAT, with various recent approaches (together with thorough bibliographies on
XORSAT) given in 
\cite{cojaoghlan2023kxorsatJournal} and \cite{chatterjee2025guidedDecimation}.
Also, siblings of XORSAT can be very challenging; for example $\littlekIntro$-SAT thresholds were not established for large $\littlekIntro$ until \cite{ding2021largeK}.

This \paper{} will prove  
that the satisfiability threshold for \KXORGAME{} equals that for  \kXORSATintro{}. 

It is culturally interesting that perfect strategies for 
XOR games seem to be studied little in the computer science community, while XORSAT is a paradigm in the field of classical computational complexity.
History sheds light on this;
XOR games arose in quantum physics.
The first XOR game to be studied is a 2XOR game now called the CHSH game. With associated experiments (in 1972), it underlays the 2022 Nobel Prize \cite{nobel2022physics}
for establishing that ``quantum entanglement" exists.

XOR games can be analyzed for both classical strategies and quantum strategies.
If a quantum strategy exists which scores better than any 
classical strategy, then it must use an entangled state, so an entangled state must exist.
For 
any 2XOR game the advantage of a quantum strategy 
over a classical strategy can be proved to be
very limited.
Subsequently, people studied 3XOR games.
3XOR games were the first proved to have a sequence of games for which the  advantage of a
  quantum strategy over a best classical strategy 
goes to infinity (\cite{perez2008unbounded,brietVidick2012boundsXOR}).

Some games have ``perfect quantum strategies"
and some do not.
For many years, it was not known if
determining which is the case for a given 
3XOR game is decidable. This was settled in  \cite{watts20203xor} where  a
polynomial time algorithm was proved to 
exist\footnote{For K-XOR having $K\geq 4$ this problem remains open.}
Then an effective algorithm together with 
computer experiments were given in \cite{watts2022satisfiability}.
The experiments motivate this paper by
suggesting that the following have satisfiability thresholds, and that the thresholds are all equal:
\begin{enumerate}
\item 
classical perfect strategies for 3XOR games,
\item 
solutions existing to 3XORSAT,
and 
\item quantum perfect strategies for 3XOR games.
\end{enumerate}
This paper confirms the first equivalence and hence is entirely classical as opposed to quantum in nature.

The introduction starts with a statement of the key linear-algebra problem over binaries, which arises when studying \KXORGAME{}s.
Ironically, we shall not actually state exactly what \KXORGAME{}s are,
since the linear algebra is simpler and contains the full mathematical issue.
Descriptions of cooperative games are available elsewhere, in particular \KXORGAME{}s
c.f. \cite{watts2018algorithms} or \cite{watts20203xor}.
The introduction continues with definitions we need, states the main theorems, and
gives brief guides to their proofs.


\subsection{Definitions of \kXORSAT{}, \KXORGAME{}, and satisfiability thresholds}

First we must clarify a point of notation.
Henceforth, we shall refer to 
 \KXORGAME{} and to \kXORSAT{}, where always
the $\littlek$ is associated with  XORSAT, and $\ourK$ is associated with XORGAME.
This contrasts with the preface where we took $k=K$ to simplify the exposition.
We shall distinguish between $k$ and $K$ because 
we will use them in some of the same places with different values.
For example, we will see that a \KXORGAME{} matrix consists of $\ourK$ blocks, where each block is a \kXORSAT{} matrix with $\littlek=1$. \cite{pittel2016satisfiability} only deals with XORSAT, so lemmas we cite from there will only use $\littlek$. 
In this paper our approach is to
patch together the analysis of $\ourK$ independent \kXORSAT{} blocks, each with $\littlek = 1$.

\subsubsection{\KXORGAME{}, \kXORSAT{}, and 2-cores.}
\label{sec:def-3-game-and-2-cores}

Fix an integer $m\geq 2$ and vector $\bn = (n_1,\ldots,n_\ourK)$ of non-negative integers. Let \begin{align*}
    \mathdf{\sigmn} = \sum_{j=1}^\ourK n_j = n_1 + \cdots + n_\ourK.
\end{align*}
A \df{\KXORGAME{} matrix} is a matrix 
$\Gamma= \As 
\in \z2^{m\times \sigmn}$,
where $A_j \in \z2^{m\times n_j}$ are blocks with $1$ one in each row (the rest of the entries being zero).
If $n_1=n_2 = \cdots = n_\ourK$, we call these \df{uniformly-tiled \KXORGAME{} matrices}.

The key problem is solving 
linear equations over the binaries:
\begin{quote}
Given a \KXORGAME{} matrix $\Gamma \in \z2^{m \times \sigmn}$
and
a vector $s \in \z2^{m}$, solve
\begin{align*}
     \Gamma x = s
\end{align*}
over $x \in \z2^{\sigmn}$.
\end{quote}
Call the pair $(\Gamma, s)$ a 
\df{game equation}.
As motivation, we state loosely that a game equation
$(\Gamma,s)$ defines a `\KXORGAME{}'.
That game has a `perfect strategy' if and only if there exists 
a solution $x$ to the linear equations $\Gamma x = s$.
For  detailed classical 
definitions (and for introduction to the quantum situation) see
\cite{perez2008unbounded} \cite{brietVidick2012boundsXOR}
or the introduction of 
\cite{watts20203xor}, which is a bit more expository.
The classical \kXORSAT{} linear systems have a similar form.
Call $\Gamma \in \z2^{m\times \sigmn} $ 
a \df{\kXORSAT{} matrix} provided each row has exactly $\littlek$ ones
on it; informally stated, there is no $\As$
partitioning.

\bs 

Next we consider degenerate cases of this linear algebra problem. Let the degree of a column of a boolean matrix be the number of ones in the column.

\ben\item 
If a column of $\Gamma$ has degree 0, then that column does not influence whether or not a solution of the equation exists; so the column could be deleted.
\item 
If a column of $\Gamma$ has degree 1, then that column corresponds to an unknown $x_j$ on which there is only one constraint; hence we could solve for
$x_j$ to eliminate that constraint. Thus we eliminate the 
column and row (including the corresponding entry in $s$) to get a new system which is solvable if and only if the original linear equations are
solvable.
\een
This leads us to define a class of matrices, which yield a non-degenerate solvability problem.

Define a \df{2-core matrix} (non-degenerate matrix) to be a matrix where each column has at least $2$ ones. 
Let the \df{2-core} of a matrix $\Gamma$, denoted $\fin \Gamma$,
be the largest submatrix of $\Gamma$ 
with the degree of every column at least 2, such that every row in $\fin \Gamma$ 
has the same number of ones as the corresponding row in $\Gamma$.
It is well-known that each matrix has one unique 2-core, as well as that the linear system $\Gamma x = s$ has a solution iff 
the corresponding 2-core equations 
$ \fin \Gamma y  = \fin s$
have a solution, see e.g. Section 7 of \cite{pittel2016satisfiability}.

\subsubsection{Randomly-generated game linear equations}

Fix a size $m,\bn= (n, \dots, n)$, and randomly generate a uniformly-tiled \KXORGAME{} matrix 
$\Gamma = \As$ with $A_j \in \R^{m\times n}$;
here we use the uniform probability on the set of \KXORGAME{} matrices. Also, generate $s \in \z2^{m}$ uniformly from the set of such vectors, to get a randomly-generated set of binary equations
$\Gamma x = s$. A goal is to understand the probability that there is a solution to these equations. This probability is heavily dependent on the ratio $\frac{m}{\ourK n }$ of constraints to unknowns.

Indeed, a dramatic piece of structure, which one sees in similar situations (\cite{dBM-FOCS}, \cite{broder1993}), \cite{pittel2016satisfiability} is there exists some \df{satisfiability threshold} $\cThreshK$ such that
\begin{itemize}
    \item If $\frac{m}{n} > \cThreshK$, then \aas{}, an $m\times \ourK n$ \KXORGAME{} problems have a solution in $\bbZ$ mod 2, and 
    \item If $\frac{m}{n} < \cThreshK$, then \aas{}, an $m\times \ourK n$ \KXORGAME{} problems have no solution in $\bbZ$ mod 2.
\end{itemize}

In this \paper{}, an event that occurs \df{asymptotically almost surely (a.a.s)} is one whose probability depends on $m,n$ and goes to 1 as $m,n\to\infty$ with a certain $\lim m/n$.

We will use {\bf big-O notation $O(\cdot)$} and {\bf little-o notation $o(\cdot)$}, typically in the context of a sequence of pairs $(m,n)$ with $m,n\to\infty$. We write $f(m,n) = O(g(m,n))$ to denote the existence of some $M>0$ such that $\abs{f(m,n)} \leq M \abs{g(m,n)}$ for all but finitely many pairs $(m,n)$. Likewise, we write $f(m,n) = o(1)$ to mean $\lim f(m,n) = 0$ as $m,n\to\infty$. Hence $f(m,n) = g(m,n)(c + o(1))$ implies $\lim f(m,n)/g(m,n) = c$.

\subsection{Main theorems}
\label{sec:mainThms}

To give a formula for the satisfiability threshold $\cThreshK$, we introduce some functions.

Define the function $Q$, strictly increasing (as in Claim 3.1 of \cite{pittel2016satisfiability}) and continuous on the reals, by
\begin{align*}
    \mathdf{Q(z)} \coloneqq \frac{z(e^z-1)}{e^z-1-z}.
\end{align*}
For each integer $K\geq 2$, let
\begin{align*}
   \mathdf{h_{\ourK}(\mu)} &\coloneqq \frac{\mu}{(e^{-\mu} (e^\mu - 1))^{\ourK - 1}}, &
    \mathdf{\cThreshK} &\coloneqq h_\ourK(Q^{-1}(\ourK)).
\end{align*}
\begin{table}[H]
    \centering
    \begin{tabular}{c||c|c|c|c|c|c|c}
        $\ourK$ & 3 & 4 & 5 & 6 & 7 & 8 & 9 
        \\ \hline
        $\cThreshK$ & 2.75381 & 3.90708 & 4.96219 & 5.98428 & 6.99345 & 7.99728 & 8.99888 
        \\ \hline
        $\cThreshK/\ourK$ & 0.917935 & 0.97677 & 0.992438 & 0.99738 & 0.999064 & 0.99966 & 0.999876 
    \end{tabular}
    \caption{Example: computed values of $\cThreshK$ for $\ourK \in \{3,4,\ldots,9\}$.}
    \label{tab:nb-computed-cThreshK}
\end{table}

\begin{theorem}
    \label{thm:main-thm-satisfiability}

    Suppose $\ourK \geq 3$. Let $m,n\to\infty$ such that $\lim m/n$ exists, and consider for each $(m,n)$ the probability that $\Gamma x = s$ is satisfiable when $(\Gamma,s)$ is generated uniformly at random on the space of uniformly-tiled \KXORGAME{} equations of this size ($m$ rows, $\ourK$ blocks of $n$ columns each).
    \begin{enumerate}
        \item If $2 < \lim m/n < \cThreshK$, then $\Gamma x = s$ is asymptotically almost surely (\aas{}) satisfiable. \label{it:main-thm-aas-sat}
        \item If $\lim m/n > \cThreshK$, then $\Gamma x = s$ is \aas{} unsatisfiable. \label{it:main-thm-aas-unsat}
    \end{enumerate}
    The satisfiability and unsatisfiability probabilities can be obtained in the proof in \Cref{sec:proofFinale}, in terms of notation used in the proof.
\end{theorem}
\begin{proof}
    The bulk of the paper works towards proving the corresponding \Cref{thm:main-theorem-core-sat} 
    which specializes to 2-core matrices.
    
    \Cref{sec:translating-2-cores} translates   \Cref{thm:main-theorem-core-sat}   on 2-cores to the original problem, culminating with the final proof in \Cref{sec:proofFinale}.
    The main part of this translation is based on \cite{botelho2012cores}, which provides the threshold for when the 2-core of a random \KXORGAME{} matrix is empty, as well as their expected sizes.
    This is presented in \Cref{lem:bwz-core-reduction}.
\end{proof}

\begin{theorem}
    \label{thm:main-theorem-core-sat}
    Suppose $\ourK\geq 3$. Let $m,n_1,\ldots,n_\ourK \to \infty$ such that $\lim m/n_j$ exists for each $j \in \{1,2,\ldots,\ourK\}$. Let $\bn = (n_1,\ldots,n_\ourK)$, and consider for each $(m,\bn)$ the probability that $\Gamma x = s$ is satisfiable when $(\Gamma,s)$ is generated uniformly at random in the space of 2-core \KXORGAME{} equations of this size ($m$ rows, $\ourK$ blocks of $n_j$ columns for the $j$th block).
    \begin{enumerate}
        \item\label{it:sat-when-wide} If $2 < \lim m/n_j < \ourK$ for $j\in \{1,2,\ldots,\ourK\}$, then $\Gamma x = s$ is asymptotically almost surely (\aas{}) satisfiable, with satisfiability probability $1-O(m^{2-\ourK})$.
        \item\label{it:unsat-when-tall} If $\lim m/n_j > \ourK$ for $j\in \{1,2,\ldots,\ourK\}$, then $\Gamma x  = s$ is \aas{} unsatisfiable, with satisfiability probability $O(2^{-(m-\sigmn)})$.
    \end{enumerate}

\end{theorem}
\begin{proof}[Proof Outline and Reader's Guide]
    This is an analog to Theorem 1.1 and part of Theorem 1.2 of \cite{pittel2016satisfiability}.

    \Cref{it:unsat-when-tall} is proven easily by the first-moment method; see Remark 1 of \cite{pittel2016satisfiability}.
    
    \Cref{it:sat-when-wide} uses the ``second-moment method" and key to implementing it here are critical row sets.
    A \df{critical row set} is a collection of rows whose sum is a vector equal to zero modulo 2.
    The proof proceeds as follows. 
    
    If $\Gamma$ is uniformly randomly generated on the space of 2-core \KXORGAME{} equations of the size $(m,\bn)$, then let the random variable $\Zmnl$ denote the number of non-empty critical row subsets of $\Gamma$ of cardinality $\ell$, and let the random variable $\mathdf{N_{m,\bn}}$ denote the number of solutions $x$ to $\Gamma x = s$.  Then the number of non-empty critical row subsets of $\Gamma$ is given by $\sum_{\ell=2}^m \Zmnl$, so
    \Cref{lem:relate-exX-and-exN} together with a second moment inequality provides
    \begin{align*}
        \mathdf{\Pr(N_{m,\bn} \geq 1)} \geq \frac{\Ex[N_{m,\bn}]^2}{\Ex[N_{m,\bn}^2]} = \frac{1}{1 + \sum_{\ell=2}^m \Zmnl}.
    \end{align*}
    \Cref{thm:sum-upper-bound-O}, whose proof
    consumes 
    \Cref{sec:bdByexpJ,sec:Jneg,sec:boundOnExpect}
    gives\allowUnusedRef{sec:bdByexpJ}\allowUnusedRef{sec:Jneg}\allowUnusedRef{sec:boundOnExpect}
    \begin{align*}
        \mathdf{\sum_{\ell = 2}^m \Ex[\Zmnl]}
        = O(m^{2-\ourK}).
    \end{align*}
    Thus
    \begin{align*}
        \Pr(N_{m,\bn} \geq 1) = \frac{1}{1 + O(m^{2-\ourK})} = 1 - O(m^{2-\ourK}),
    \end{align*}
    which tends to 1 as $m \to \infty$.
    Our detailed proof which fills in the outline just given concludes in \Cref{sec:proofFinale}.

This proof follows the outline in \cite{pittel2016satisfiability}
    but with issues and  formulas which are different in substantial ways causing their own concerns.
    
\end{proof}

\subsection{Hypergraph interpretation}

There are two ways to view and  state the main results of this paper: one is with matrices, and another is with hypergraphs. The predominate perspective we use here is with matrices, although late in the paper,
\Cref{sec:translating-2-cores}, we switch to hypergraph 
language, to be easily compatible with 
the terminology in \cite{botelho2012cores}.

Of possible interest, we shall additionally give a new proof of the ``Maintenance of Uniformity'' principle.

\section{Definitions of random variables related to critical row sets}

Here, we formally define 
the random variables $\Ymnkl$ and $\Zmnl$ used throughout the paper.

\subsection{Critical row sets vs second moment theorem}

Now comes a basic and powerful observation.

\begin{lemma}
    \label{lem:relate-exX-and-exN}    
    Let $A\in \{0,1\}^{m\times n}$ be a random variable with any given distribution, and let $s$ be independent of $A$ and uniformly distributed over $\{0,1\}^m$. Let $N$ denote the number of binary solutions to $Ax = s$, and let $X$ denote the number of non-empty critical row subsets of $A$. Then \begin{align*}
        \Ex[N^2]/\Ex[N]^2 = \Ex[X]+1.
    \end{align*}
\end{lemma}
\begin{proof}
    This is Remark 2 from Pittel-Sorkin \cite{pittel2016satisfiability}.
\end{proof}

Thus $\Ex[X] \to 0$ is sufficient to show $\Ex[N^2]/\Ex[N]^2 \to 1$. This connects the second moment method with the critical set approach of Kolchin \cite{kolchinRandomGraphs}. To apply this method, we first need to introduce some formal notation.

\subsection{Re-definition of game matrix in terms of $\cA$}

\label{sec:def-Psimn}

Let $\cA_{m,n,k}$ be the set of all $m\times n$ boolean 2-core matrices with $k$ ones per row. For $\cA_{m,n,k}$ to be non-empty, it is necessary that $km \geq 2n$. 
To see this observe each
matrix in $\cA_{m,n,k}$ has $m$ rows each with $k$ one, so it has exactly $km$ ones total in the matrix. At the same time, each of the $n$ columns must have at least $2$ ones, so the matrix must have at least $2n$ ones total. Hence $km\geq 2n$.

Fix $\bn = (n_1,\ldots,n_\ourK)$. Let $\mathdf{\Psimn}$ be the set of \df{2-core \KXORGAME{} matrices} 
which are block matrices  $\Gamma = (A_1\ \cdots\ A_\ourK)$ with blocks $A_j\in \cA_{m,n_j,1}$ for $j \in \{1,2,\ldots,\ourK\}$. Note here we set $\littlek=1$ since each block needs exactly $1$ one per row.
To ensure $\Psimn$ is nonempty, $m\geq 2n_j$ must hold (for each $j\in \{1,\ldots,\ourK\}$) due to the similar constraint to ensure $\cA_{m,n,k}$ is nonempty.
To ensure
$\Psimn$ is nonempty, $m\geq 2n_j$ must hold (for each $j\in \{1,\ldots,\ourK\}$) due to the similar constraint to ensure $\cA_{m,n,k}$ is nonempty.

We shall consider 
probability spaces
$\cA_{m,n,k}$ and $\Psimn$ with the uniform distribution.

\subsection{Definition of $X$ and the random variables $\Ymnkl$ and $\Zmnl$}
\label{sec:random-variables-YZ}

For a matrix $\Gamma$, let $X(\Gamma)$ be the number of non-empty critical row subsets of $\Gamma$.
Furthermore, for $\ell \geq 1$, let $\mathdf{X^{(\ell)}(\Gamma)}$ denote the number of critical row subsets of $\Gamma$ with cardinality $\ell$. If every row of $\Gamma$ is nonzero, then $X^{(1)}(\Gamma) = 0$, and 
\begin{align*}
    \mathdf{X(\Gamma)} = \sum_{\ell=2}^m X^{(\ell)}(\Gamma).
\end{align*}

For $m,n,k\geq 1$ with $km \geq 2n$ and $n \geq k$, 
define the random variable
$\Ymnkl\colon \cA_{m,n,\littlek} \to \bbZ_{\geq0}$
by 
\begin{align*}
    \mathdf{\Ymnkl}(A) = X^{(\ell)}(A) 
    \foralldisplay{A \in \cA_{m,n,\littlek}, \ell \in \{1,\ldots,m\}},
\end{align*}
that is,
the number of critical row sets of $A$ of cardinality $\ell$.

Likewise, fix $m,\bn = (n_1,\ldots,n_\ourK)$ with $m \geq 2n_j$ and  $\ell \in \{1,\ldots,m\}$.
Define the random variable $\Zmnl\colon \Psimn \to \bbZ_{\geq0}$ 
by
\begin{align*}
    \mathdf{\Zmnl}(\Gamma) = X^{(\ell)}(\Gamma) 
    \foralldisplay{\Gamma \in \Psimn},
\end{align*}
that is,
the number of critical row sets of $\Gamma$ of cardinality $\ell$.

\section{Upper bounds on expected number of critical row sets in a 2-core}
\label{sec:bdByexpJ}

This section is devoted to producing an upper bound on 
$\Ex[\Zmnl]$.
This bound is an analog of 
Lemma 4.1 of \cite{pittel2016satisfiability}.
Our proof follows its approach, with somewhat different functions and additional
arguments.
For $c > 2$, $\alpha \in [0,1]$, and $\zzz = (\zeta_1,\zeta_2) \in \R_{>0}^2$, define $\mathdf{\JK(\alpha,\zzz;c)}$ by
\begin{align}
    \begin{split}
        \mathdf{\JK(\alpha,\zzz;c)} &= \frac{1}{\ourK} H(\alpha) + \alpha \ln(\alpha/\zeta_1) + \ov\alpha \ln(\ov\alpha / \zeta_2)
        \\&\ + \frac{1}{c} \ln \frac{\expMTwo(\lambda \cdot (\zeta_2 + \zeta_1)) + \expMTwo(\lambda \cdot (\zeta_2 - \zeta_1))}{2 \expMTwo(\lambda)},
    \end{split}
    \label{def:Jkm}
\end{align}
where
\begin{align*}
    \mathdf{H(\alpha)} &\coloneqq -\alpha \ln \alpha - \ov\alpha\ln \ov\alpha, \quad & \mathdf{\ov\alpha} &= 1-\alpha
    \\ \mathdf{\expMTwo(z)} & \coloneqq e^z - 1 - z, & \quad \mathdf{\lambda} &\coloneqq Q^{-1}(c).
\end{align*}
Consistent with continuity, we define $0\ln(0)=0$. 
Throughout this section, we take $O(1)$ to mean some constant that may depend on the choice of function $\zzz$, but does not depend on $m,n,k$ or $\ell$.

The bound we shall prove is stated next.

\begin{proposition}
\label{prop:ZbdJ}
    Suppose $m,n_1,\ldots,n_\ourK\to \infty$ with $\lim m/n_i \in (2, \infty)$ for $i\in \{1,\ldots,\ourK\}$. For $\ell \in \{1,\ldots,m\}$, 
    with $c_i=m/n_i, \alpha=\ell/m, $ and any piecewise continuous function $\zzz \colon [0,1] \to \R_{>0}^2$ written as $\mathdf{\zzz(\alpha)} = (\zeta_1(\alpha),\zeta_2(\alpha))$,
    we have
    \begin{align*}
        \Ex[\Zmnl] &\leq O(1) \frac{(\ell)^{(\ourK-1)/2}}{\sqrt{(\zeta_2(\alpha))^\ourK}} \exp(m \sum_{i=1}^\ourK \JK(\alpha,\zzz(\alpha);c_i)).
    \end{align*}
\end{proposition}

\subsection{Proof of this bound on 
 $\Ex[\Zmnl] $}
We begin our proof by stating 
a slight strengthening of 
\cite{pittel2016satisfiability} Lemma 4.1 derived 
for \kXORSAT. 
It gives the upper bound on 
the expectation $\Ex[Y_{m,n,\littlek}^{(\ell)}]$ of the critical set random variable.

To state it, we first define $\Hk$,
for $\alpha \in [0,1]$, $c>2$, and any $\zzz=(\zeta_1,\zeta_2)$ \index{$\zeta_1$,  $\zeta_1$} with $\zeta_1>0$ and $\zeta_2>0$, by 
\begin{align}
\begin{split}
    \Hk(\alpha,\zzz;c) &\coloneqq cH(\alpha) + c\littlek (\alpha \ln(\alpha/\zeta_1) 
    + c\littlek  (\ovalpha) \ln((\ovalpha ) / \zeta_2)
    \\&\ + \ln \frac{\expMTwo(\lambda \cdot (\zeta_2 + \zeta_1)) + \expMTwo(\lambda \cdot (\zeta_2 -  \zeta_1))}{2 \expMTwo(\lambda)},
    \\ \lambda &\coloneqq Q^{-1}(ck).
    \label{def:Hk}
\end{split}
\end{align}
This is the same as the definition for $\Hk$ in Equation 4.2 of \cite{pittel2016satisfiability}, though there are the notational differences: our $\expMTwo$ is their $f$ and our $Q$ is their $\psi$.

\begin{lemma}
    \label{lem:initial-upper-bound-exY-in-m}
    Suppose $k \geq 1$ and $m,n\to \infty$ 
    with $\lim m/n \in (2/k, \infty)$. For $\ell \in \{1,\ldots,m\}$ and $\mathdf{\alpha} = \ell/m$, with any piecewise continuous function $\zzz \colon [0,1] \to \R_{>0}^2$ written as $\mathdf{\zzz(\alpha)} = (\zeta_1(\alpha),\zeta_2(\alpha))$, we have
    \begin{align*}
        \mathdf{\Ex[Y_{m,n,k}^{(\ell)}]} \leq O(1) \frac{1}{\sqrt{\zeta_2(\alpha)}} \exp[n \Hk(\alpha,\zzz;m/n)].
    \end{align*}
\end{lemma}

\begin{proof}
    The lemma statement here is almost precisely the same as the result of Lemma 4.1 of \cite{pittel2016satisfiability}. 
    The only difference \Cref{lem:initial-upper-bound-exY-in-m} makes is weakening their $\littlek \geq 3$ assumption to $\littlek \geq 1$. The proof in \cite{pittel2016satisfiability} works equally well with $\littlek \geq 1$ as it does with $\littlek \geq 3$, so their argument suffices to prove this lemma.
\end{proof}

We have introduced $\JK$ above to serve for a single block of a \KXORGAME{} matrix; note $\JK$ is a shifted and scaled form of $\Hk$ with $k=1$, precisely
\begin{align}
    \label{eq:Jk-in-H1}
    \JK(\alpha,\zzz;c) = \frac{1 - \ourK}{\ourK} H(\alpha) + \frac{1}{c} \Hone(\alpha,\zzz;c).
\end{align}

\noindent 
{ \it Proof of \Cref{prop:ZbdJ}.}
    The first $\ell$ rows of a random matrix in $\cA_{m,n,1}$ form a critical row set with probability $
        \mathdf{P_{m,n,1}(\ell)} \coloneqq \frac{\Ex[Y_{m,n,1}^{(\ell)}]}{\binom{m}{\ell}}.
   $
    Hence, the first $\ell$ rows of a random matrix in $\Psimn$ form a critical row set with probability $
        \prod_{r=1}^{\ourK} P_{m,n_r,1}(\ell).
    $
    There are $\binom{m}{l}$ sets of rows, each with the same probability of forming a critical row set, so by linearity of expectation,
    \begin{align*}
        \mathdf{\Ex[\Zmnl]} &= \binom{m}{\ell} \prod_{i=1}^{\ourK} P_{m,n_r,1}(\ell)
        = \binom{m}{\ell} \prod_{i=1}^{\ourK} \frac{\Ex[Y_{m,n_i,1}^{(\ell)}]}{\binom{m}{\ell}}
        \\ &= \binom{m}{\ell}^{1-\ourK} \prod_{i=1}^{\ourK} \Ex[Y_{m,n_i,1}^{(\ell)}].
    \end{align*}
    Applying \Cref{lem:initial-upper-bound-exY-in-m} yields, for each $i\in\{1,\ldots,\ourK\}$
    \begin{align*}
        \mathdf{\Ex[Y_{m,n_i,1}^{(\ell)}]} \leq O(1) \frac{1}{\sqrt{\zeta_2(\alpha)}} \exp(n_i \Hone(\alpha,\zzz;m/n_i)).
    \end{align*}
    Hence
    \begin{align*}
        \mathdf{\Ex[\Zmnl]} &\leq \binom{m}{\ell}^{1-\ourK} \prod_{i=1}^{\ourK} O(1) \frac{1}{\sqrt{\zeta_2(\alpha)}} \exp(n_i \Hone(\alpha,\zzz;m/n_i)).
        \\ &\leq \binom{m}{\ell}^{1-\ourK} \frac{O(1)}{\sqrt{(\zeta_2(\alpha))^\ourK}} \exp(\sum_{i=1}^\ourK n_i \Hone(\alpha,\zzz;m/n_i)).
    \end{align*}
    For all $m\geq 1$ and $1\leq \ell \leq m$, we get
    \begin{align*}
        \binom{m}{\ell}^{-1} \leq O(1) \sqrt{\ell} \exp(-mH(\alpha)).
    \end{align*}
    where $\alpha = \ell/m$ and $\mathdf{H(\alpha)}\coloneqq -\alpha \ln(\alpha) - (1-\alpha)\ln(1-\alpha)$ is the entropy function. This is a consequence of Stirling's approximation; for details of the proof, see \Cref{lem:binomial-inv-upper-bound-appendix}.
    Hence
    \begin{align*}
        \binom{m}{\ell}^{1-\ourK} \leq O(1) \ell^{(\ourK-1)/2} \exp(-m(\ourK - 1)H(\alpha)),
    \end{align*}
    so
    \begin{align*}
        \Ex[\Zmnl] &\leq O(1) \frac{(\ell)^{(\ourK-1)/2}}{\sqrt{(\zeta_2)^\ourK}} \exp(m \left((1-\ourK)H(\alpha) + \sum_{i=1}^{\ourK} \frac{n_i}{m} \Hone(\alpha,\zzz;m/n_i)\right)).
    \end{align*}
    Substituting the definition of $\JK$ from \Cref{eq:Jk-in-H1} gives the result,
    \begin{align*}
        \Ex[\Zmnl] &\leq O(1) \frac{(\ell)^{(\ourK-1)/2}}{\sqrt{(\zeta_2(\alpha))^\ourK}} \exp(m \sum_{i=1}^\ourK \JK(\alpha,\zzz(\alpha);m/n_i)).
    \end{align*}
    
\qed

\section{Proving that $\JK $ is negative}
\label{sec:Jneg}

Our overarching 
goal is to prove that $\Ex\left[\sum_{\ell=2}^m \Zmnl\right]$ goes to zero in certain parameter ranges; this is finalized in 
\Cref{thm:sum-upper-bound-O}. To apply the upper bound on this expectation given by
\Cref{prop:ZbdJ}, we need to show 
an inequality 
slightly stronger 
than just $\JK < 0$ holds for $\alpha \in (0,1]$. The needed assertion is in 
\Cref{prop:Jk-upper}, and much of this section works toward proving that proposition.

\subsection{General idea and Reader's Guide}

The analysis of $\JK$ as a function of $\alpha$ in $(0, 1]$ will be split into primarily two intervals $(0,0.99\betaK] \cup [\betaK/2,\,1]$,
where
\begin{align}
    \label{def:beta-k}
    \mathdf{\betaK} &\coloneqq \exp(-\left( \frac{1}{\ourK} + \ln(\sqrt{\ourK-1}) - 1 + \frac{\ourK}{2(\ourK-1)}\right)\Big/\left(\frac{1}{2} - \frac{1}{\ourK}\right)).
\end{align}
One can check that $\betaK < 0.2$ 
for all $\ourK\geq 3$.
The proof is easy but we included it, see
\Cref{sec:proof-betaK-upper-bound}.
Note $\betaK$ is analogous to $\alpha_k$ from Pittel-Sorkin, but $\betaK$ is (slightly) smaller,
which causes us the difficulty that 
we must handle the interval 
 $[\betaK, \alpha_\ourK)$ separately.

\Cref{prop:Jk-upper}, which we are working toward, provides bounds on $\JK(\alpha,\zzz(\alpha),c)$ holding for $\alpha\in(0,1]$ when $\zzz$ is defined as a particular piecewise function. Most branches of the piecewise function are related to $\zzzlin$ or $\zzzsqrt$, which we define as
\begin{align}
    \mathdf{\zzzlin(\alpha)} &= (\zeta_1(\alpha), \zeta_2(\alpha)) \quad \text{with} \quad \zeta_1(\alpha) = \alpha, &&\quad \zeta_2(\alpha) = 1-\alpha \label{def:eq-zzzlin}
    \\ \mathdf{\zzzsqrt(\alpha)} &= (\zeta_1(\alpha), \zeta_2(\alpha)) \quad \text{with} \quad \zeta_1(\alpha) = \sqrt{\alpha/(\ourK-1)}, &&\quad \zeta_2(\alpha) = 1-\alpha. \label{def:eq-zzzsqrt}
\end{align}

The proof involves theory, results invoked from \cite{pittel2016satisfiability},
and interval arithmetic (numerical calculations) in the Mathematica notebook \cite{nbIntervalCrit}.

\subsubsection{Outline of approach to bounding $\JK$}

This section gives bounds on 
$\JK(\alpha,\zzz(\alpha),c)$ along a particular curve $\zzz$. This is concluded 
in \Cref{prop:Jk-upper}, 
proved
by implementing the following outline:

\ben 
\item For all $K\geq 3$ and $c \in (2,K)$,
\Cref{lem:Jk-upper-near-0} uses $\zzz = \zzzsqrt$ to give a good bound on $\JK$ for 
$\alpha \in (0,\betaK)$. 

\item For $\ourK = 3$, the same choice ($\zzz = \zzzsqrt$) 
is enough to show $\JK < -\epsK(c)$ for the remaining $\alpha \in [0.99\betaK, 1/2]$, for all $c\in (2,\ourK)$. Completing this proof in \Cref{lem:nb-K3-Jsqrt-bound} requires rigorous numerical calculations which invoke interval arithmetic.

\item For $\ourK \geq 4$, the choice $\zzz = \zzzsqrt$ does not suffice, so we pick $\zzz = \zzzlin$. We begin by obtaining $\JK \leq 0$ when $c = \ourK$ in \Cref{lem:nb-Kgeq4-J-bound}, with a combination of interval arithmetic and theory from \cite{pittel2016satisfiability}.

On $c \in (2,\ourK)$, the result strengthens to $\JK \leq -\epsK(c)$, as obtained in \Cref{lem:Kgeq4-J-bound-allc}. The proof uses a conditional monotone dependency of $ c \JK(\alpha,\zzzlin(\alpha) ;c)$ on $c$, shown in \Cref{lem:monotonicity-in-c}.

\item 
For all $\ourK\geq 3$, a continuity argument in \Cref{lem:deltaBallK} provides the existence of $\deltaBallK$ such that $\JK \leq -\epsK(c)$ on $\alpha \in [1-\deltaBallK(c),1]$ when we choose $\zzz(\alpha) = (1-\deltaBallK(c),\deltaBallK(c))$.

A reflection argument in \Cref{lem:Jk-upper-near-1} extends $\JK \leq -\epsK(c)$ from $(\deltaBallK(c),1/2]$ to also apply on $\alpha \in [1/2,1 - \deltaBallK(c))$. The choice of $\zzz$ here is a reflection of a piecewise function between $\zzzsqrt$ and $\zzzlin$, as described in \Cref{eq:monster-zzz}.

\een

Combining each item above shows the explicit function $\zzz$ defined in \Cref{eq:monster-zzz} produces $\JK(\alpha,\zzz(\alpha);c)$ which satisfies a good bound for $\alpha \in (0,1)$.

\bs

\subsection{Bounds along $\zzzlin$.}
\label{sec:nb-bounds-along-zzzlin}

In this section, we will evaluate $\JK$ at $\zzzlin$ as defined in \Cref{def:eq-zzzlin}, so define $\LK$ by
\begin{align}
    \label{def:LK}
    \mathdf{\LK(\alpha,c)} \coloneqq c \JK(\alpha,\zzzlin(\alpha);c)
\end{align}
We will want certain estimates on $\LK$. 

To do this, we start with some results of 
Pittel-Sorkin on the function $\Hk(\alpha,\lambda)$, which we now define as 
\begin{align*}
    \mathdf{\Hk(\alpha;\lambda)} \coloneqq \Hk(\alpha,\zzzlin(\alpha);c) \quad \mathrm{with} \ \ 
    c = \frac{1}{\littlek} Q(\lambda),
\end{align*}
in terms of the three-argument function $\Hk$ defined in \Cref{def:Hk}. This matches the definition of the two-argument $\Hk$ in Equation 5.8 of \cite{pittel2016satisfiability}. We now relate this function to $\LK(\alpha,c)$.

\begin{lemma}
    \label{lem:Hk-equals-LK}
    If $\littlek = \ourK$, then
    \begin{align*}
        \Hk(\alpha;\lambda) = \LK(\alpha,Q(\lambda)).
    \end{align*}
\end{lemma}
\begin{proof}
    Expanding the definitions of $\Hk$ and $\LK$ yields
    \begin{align*}
        \mathdf{\Hk(\alpha;\lambda)} &= \frac{Q(\lambda)}{\littlek} H(\alpha) + \ln \frac{\expMTwo(\lambda) + \expMTwo(\lambda \cdot (1 - 2\alpha))}{2 \expMTwo(\lambda)}
        \\ \mathdf{\LK(\alpha,c)} &= \frac{c}{\ourK} H(\alpha) + \ln \frac{\expMTwo(\lambda) + \expMTwo(\lambda \cdot (1-2\alpha))}{2 \expMTwo(\lambda)}.
    \end{align*}
    The result follows from $c=Q(\lambda)$ and $\littlek = \ourK$.
\end{proof}

\begin{lemma}
    \label{lem:monotonicity-in-c}
    For all $\ourK \geq 4$, when $\LK(\alpha,c) \geq 0$, then $\LK(\alpha,c)$ is strictly increasing in $c$.
\end{lemma}
\begin{proof}
    This is Equation 5.10 from the proof of Claim 5.2 of \cite{pittel2016satisfiability}, translated into our notation. They show $\Hk(\alpha;\lambda)$ is strictly increasing in $\lambda$ when $H_k(\alpha;\lambda)\geq 0$. Since $Q$ is strictly increasing, the equality $H_k(\alpha,\lambda) = \LK(\alpha,Q(\lambda))$ from \Cref{lem:Hk-equals-LK} shows $\LK(\alpha,c)$ is strictly increasing in $\lambda$ when $\LK(\alpha,c)\geq 0$. 
\end{proof}

To this point, all our results have come from analytical techniques. Now, we turn to numerical (floating-point) evaluations made rigorous by interval arithmetic.
This is done in the Mathematica notebook \cite{nbIntervalCrit} 
and summarized in appropriate places here. The next lemma uses interval arithmetic on a function of one variable, while later we use it 
on functions of 2 variables.
For a fixed value of $\ourK$ and an interval $[\alpha_0,\alpha_1]$, the technique of interval arithmetic allows us to determine an interval within which the value of $L_K(\alpha,c)$ provably must lie for all $\alpha \in [\alpha_0,\alpha_1]$.

For each $\littlek \geq 3$, define, as in \cite{pittel2016satisfiability},
\begin{align*}
    \mathdf{\alpha_k}
    = ek^{-\littlek/(\littlek-2)},
    \qquad 
     \mathdf{\alpha_k^*} = \frac{1}{2}\left(1 - \frac{1}{\lambda_k} \ln \frac{\expMTwo(\lambda_k)}{\lambda_k^2}\right), \qquad
    \mathdf{\lambda_k} = Q^{-1}(\littlek).
\end{align*}
Note $\alpha_k^* \in (0,1/2)$ is confirmed after Equation 5.23 of \cite{pittel2016satisfiability}. Likewise, $\alpha_k \in (0,0.2)$ is straightforward in a similar way to the approach to $\betaK < 0.2$ from 
after \Cref{def:beta-k}.

\begin{lemma}
    \label{lem:nb-Kgeq4-J-bound}
    For $\ourK \geq 4$,
    \begin{align*}
        \LK(\alpha,\ourK) \leq 0 \foralldisplay{\alpha \in [0.99 \betaK,1/2]}.
    \end{align*}
\end{lemma}
\begin{proof}
    This is an analog to ``Case $\alpha$ near 1/2'' from the proof of Claim 5.2 of \cite{pittel2016satisfiability}.
    
    A claim before Equation 5.23 of \cite{pittel2016satisfiability} shows $\Hk(\alpha,Q^{-1}(\littlek)) \leq 0$ for all $\alpha \in [\alpha_\littlek^*,1/2]$ for all $\littlek \geq 3$. Hence for all $\ourK\geq 3$,
    \begin{align*}
        \LK(\alpha,\ourK) \leq 0 \foralldisplay{\alpha \in [\alpha_\ourK^*,1/2]}.
    \end{align*}
    Thus it remains for us to show the claim on $\alpha \in [0.99\betaK, \alpha_\ourK^*]$.

    For all $\ourK\geq 7$, \Cref{lem:nb-Lk-neg-K-geq-7} shows
    \begin{align*}
        \LK(\alpha,\ourK) \leq 0 \foralldisplay {\alpha \in [0.99\betaK, \alpha_\ourK^*]}.
    \end{align*}
    For $\ourK=4,5,6$, we use a interval arithmetic on a uniform grid of 25 sub-divisions (see \cite{nbIntervalCrit}) to show
    \begin{align}
        \label{eq:nb-1-LK-neg}
        \LK(\alpha,\ourK) \leq -10^{-4} \leq 0 \foralldisplay{\alpha \in [0.15,0.45]}.
    \end{align}
    For $\ourK=4,5,6$, direct evaluation yields 
    \begin{align}
        \label{eq:nb-1a-beta-alpha-fits}
        [0.99\betaK, \alpha_\ourK^*] \subseteq [0.15,0.45],
    \end{align}
    completing the proof.
\end{proof}

The result is illustrated for $\ourK\in \{4,5,6\}$ in \Cref{fig:nb-Jklin456}.
\begin{figure}[H]
    \centering
    \includegraphics[width=0.7\linewidth]{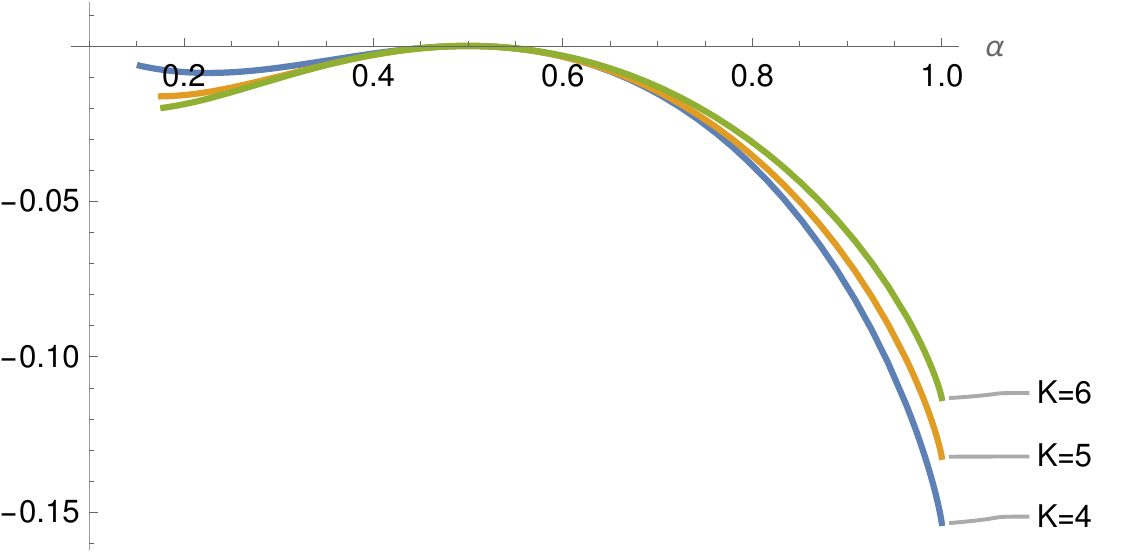}
    \caption{$\JK(\alpha,\zzzlin;c) = \frac{1}{c} \LK(\alpha,c)$ plotted over $\alpha \in [\betaK,1]$ for $c=\ourK \in \{4,5,6\}$.}
    \label{fig:nb-Jklin456}
\end{figure}

\begin{lemma}
    \label{lem:nb-Lk-neg-K-geq-7}
    If $\ourK \geq 7$, then
    \begin{align*}
        \LK(\alpha,\ourK) < -0.016 \foralldisplay{\alpha \in [0.99\betaK, \alpha_\ourK^*]}.
    \end{align*}
\end{lemma}
\begin{proof}
    If $0 < \alpha \leq \alpha_\ourK^* \leq \half$, then from Equation 5.24 
    of \cite{pittel2016satisfiability},
    \begin{align}
        \label{eq:ps-24-LK-bound-initial}
        \LK(\alpha,\ourK) \leq H(\alpha) + \ln \frac{1 + e^{-2 \ourK \alpha}}{2}.
    \end{align}
    Let $\alphaloK = 0.99\betaK$ and $\alphahiK = \alpha_\ourK^*$. Note $0 < \betaK < \alpha_\ourK^* < \half$.

    \cite{pittel2016satisfiability} shows $\alphahiK = \alpha_\ourK^*$ is decreasing in $\ourK$, so $\alphahiK \leq \alpha_7^*$.
    Differentiation of $\ln(\ourK \betaK)$ shows $\ourK \betaK$ and hence $\ourK \alphaloK = 0.99\ourK \betaK$ is increasing in $\ourK \geq 7$. Hence $\ourK \alphaloK \geq 0.99\cdot 7 \beta_7$ for $\ourK \geq 7$.
    
    Thus, for $\ourK \geq 7$ and $\alpha \in [\alphaloK, \alphahiK]$, we have $\alpha \leq \alpha_7^* < \half$ and $\ourK \alpha \geq 0.99 \cdot 7\beta_7$. The first term, $H(\alpha)$, in \Cref{eq:ps-24-LK-bound-initial} is decreasing for $\alpha < \half$, and the second term is increasing in $\alpha$, so
    \begin{align}
        \LK(\alpha,\ourK) \leq H(\alpha_7^*) + \ln \frac{1 + e^{-2\cdot 0.99\cdot 7\beta_7}}{2} < -0.016 \foralldisplay{\alpha \in [\alphaloK, \alphahiK]}. \label{eq:nb-upper-bound-LK}
    \end{align}
    where the second inequality is by direct calculation in \cite{nbIntervalCrit}.
\end{proof}

\begin{figure}[H]
    \centering
    \includegraphics[width=0.7\linewidth]{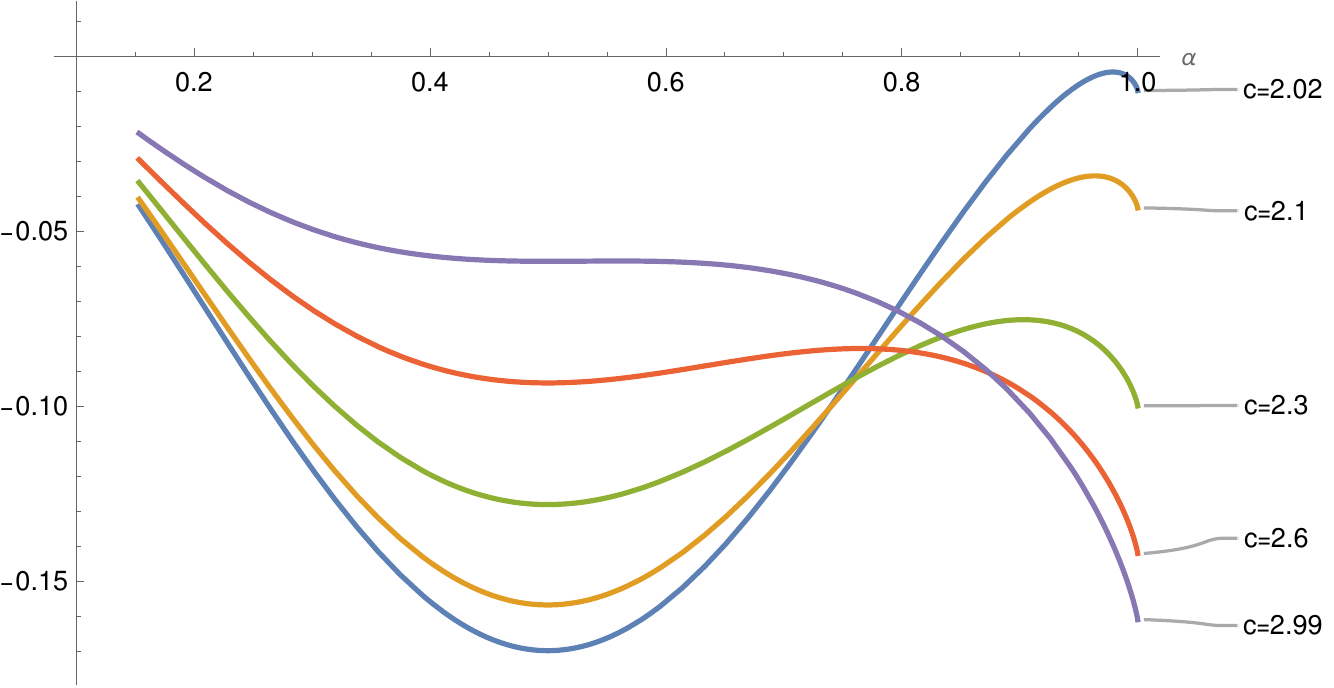}
    \caption{$\JK(\alpha,\zzzlin(\alpha);c)$ plotted for $c=2.02,2.1,2.3,2.6,2.99$ and $\ourK=4$ over $\alpha \in [0.99\betaK,1]$.}
    \label{fig:nb-JKLin_K4_varying_c}
\end{figure}

\begin{lemma}
    \label{lem:Kgeq4-J-bound-allc}
    For $K\geq 4$, there exists a positive function $\epsK$ continuous in $c\in (2,\ourK)$ such that
    \begin{align*}
        \JK(\alpha,\zzzlin(\alpha);c) \leq -\epsK(c) \foralldisplay{\alpha \in [0.99\betaK, 1/2],\, c\in (2,\ourK)}.
    \end{align*}
\end{lemma}
This result is illustrated in \Cref{fig:nb-JKLin_K4_varying_c}.
\begin{proof}
    This argument is based on the discussion around Equation (5.17) in the proof of Claim 5.2 from \cite{pittel2016satisfiability}.
    
    Recall \Cref{def:LK} defines
    $
        \mathdf{\LK(\alpha,c)} \coloneqq c \JK(\alpha,\zzzlin(\alpha);c).
    $
    For $c \in (2,\ourK)$, let 
    \begin{align*}
        \mathdf{m_\ourK(c)} \coloneqq \sup\{\LK(\alpha,c_0) \mid \ \alpha \in [0.99\betaK, 1/2], 
    \ c_0 \in [2,c]\}
    \end{align*}
    By continuity of $\LK$, the supremum is attained at some $(\hat\alpha,\hat c_0)$, where $0.99\betaK \leq \hat\alpha \leq 1/2$ and $\hat c_0 \in [2,c] \subseteq [2,\ourK)$, so $\hat c_0 < \ourK$.    
    By \Cref{lem:monotonicity-in-c}, if $m_\ourK(c) = \LK(\hat\alpha,\hat c_0) \geq 0$, then $\LK(\hat\alpha,c_0)$ is strictly increasing in $c_0$ for all $c_0 \in [\hat c_0,\ourK]$, so $\LK(\hat\alpha,\ourK) > 0$. However, by \Cref{lem:nb-Kgeq4-J-bound} we have $\LK(\hat\alpha,\ourK) \leq 0$, so by contradiction,
    \begin{align*}
        m_\ourK(c) = \LK(\hat\alpha,\hat c_0) < 0.
    \end{align*}

    Note $m_\ourK(c)$ is continuous in $c$ by continuity of $L_\ourK$, so $\epsK(c) \coloneqq -\frac{1}{\ourK} m_\ourK(c)$ is continuous in $c$. Then by the  definition of $m_\ourK(c)$
    as a supremum,
    \begin{align*}
        \LK(\alpha,c) \leq m_\ourK(c) = -K\epsK(c) \foralldisplay{\alpha \in [0.99\betaK, 1/2],\, c \in (2,\ourK)}.
    \end{align*}
    Then since $-\ourK\leq -c$,
    \begin{align*}
        \ourK \JK(\alpha,\zzzlin(\alpha);c) \leq c \JK(\alpha,\zzzlin(\alpha);c) = \LK(\alpha,c) \leq -\ourK \epsK(c).
    \end{align*}
    Dividing by $\ourK$ finishes the proof.
\end{proof}

\subsection{Bounds along $\zzzsqrt$.}

In this section, we will evaluate $\JK$ at $\zzzsqrt$ as defined in \Cref{def:eq-zzzsqrt} and prove it negative in two different regions. 
\Cref{lem:nb-K3-Jsqrt-bound} applies for $K=3$ and $\alpha$ away from $0$, while \Cref{lem:Jk-upper-near-0} applies for all $\ourK\geq 3$ and $\alpha$ near $0$.

\begin{lemma}
    \label{lem:nb-K3-Jsqrt-bound}
    For $K=3$, there exists a positive function $\epsilon$ continuous in $c\in (2,3)$ such that
    \begin{align*}
        J_3(\alpha,\zzzsqrt(\alpha);c) \leq -\epsilon(c) \foralldisplay{\alpha \in [0.99\beta_3, 1/2],\, c\in (2,3)}.
    \end{align*}
\end{lemma}
\begin{proof}
    \def\Jsqrt3{J_{sqrt,3}}
    Letting $\ourK=3$ and $\zzz = \zzzsqrt$, define $\Jsqrt3$ as 
    \begin{align*}
        \mathdf{\Jsqrt3(\alpha,\lambda)} = J_3(\alpha,\zzzsqrt(\alpha);Q(\lambda)),
    \end{align*}
    where $\JK$ was defined in \Cref{def:Jkm}. For all $\lambda$, the formula heavily simplifies at $\alpha = \half$ as verified in \cite{nbIntervalCrit}, yielding
    \begin{align}
        \label{eq:nb-4-Jsqrt3-at-half}
        \Jsqrt3\left(\half,\lambda\right) &= \left(\frac{1}{3}-\frac{1}{Q(\lambda)}\right)\ln(2),
        & \mypdv{\Jsqrt3}{\alpha}\left(\half,\lambda\right) &= 0.
    \end{align}

    We turn again to 
    interval arithmetic, 
    see \cite{nbIntervalCrit}, and evaluate the functions $\Jsqrt3$ and $\mypdv{^2 \Jsqrt3}{\alpha^2}$ on 
    rectangular subsets $(\alpha,\lambda) \subseteq [0,1]\times [0,\infty)$; our Mathematica notebook corresponds to the presentation here closely enough to be readable.
    See \Cref{fig:zzzsqrt-regions} for a diagram of these rectangles and see the verbal discussion of each one.

    \begin{figure}[H]
        \centering
    \begin{tikzpicture}[x=2.4cm,y=2.4cm]
        \def\xa{0}\def\xb{0.75}\def\xc{1.5}\def\xd{3}\def\xe{4}
        \def\ya{0}\def\yb{2}\def\yc{2.75}\def\yd{3}
        \def\crad{0.15}
        \def\fontscale{1.2}
        \def\colA{black}\def\colB{black}\def\colC{black}
        \def\green{green!50!black}
        
        \coordinate (p2b) at ($(\xb,\yb)!0.5!(\xd,\yd)$);
        \coordinate (p3) at ($(\xd,\yb)!0.5!(\xe,\yd)$);
        \coordinate (p2a) at ($(\xb,\ya)!0.5!(\xe,\yb)$);

        \draw[] (\xa,\ya) rectangle (\xe,\yd);
        \draw[draw=none,pattern={Lines[angle=45,distance=4pt]},pattern color=\colB] (\xb,\yb) rectangle (\xd,\yd) (p2b) circle (\crad);
        \draw[draw=none,pattern={Lines[angle=-45,distance=4pt]},pattern color=\colC] (\xd,\yb) rectangle (\xe,\yd) (p3) circle (\crad); 
        \draw[draw=none,pattern={Lines[angle=0,distance=4pt]},pattern color=\colA] (\xb,\ya) rectangle (\xe,\yb) (p2a) circle (\crad); 
        
        \draw[\colA,thick,opacity=0.9] (\xb,\ya)--(\xb,\yd);
        \draw[\colB,thick,opacity=0.9] (\xb,\yb)--(\xe,\yb);
        \draw[\colC,thick,opacity=0.9] (\xd,\yb)--(\xd,\yd);
        
        \draw[\green,line width=4pt] (\xc,\ya) rectangle (\xe,\yc);
        

        \node[scale=\fontscale,\colA] at (p2a) {$2a$};
        \node[scale=\fontscale,\colB] at (p2b) {$2b$};
        \node[scale=\fontscale,\colC] at (p3) {$3$};

        \node[scale=\fontscale,anchor=north] at ($(\xa*0.5+\xe*0.5,\ya-0.3)$) {$\alpha$};
        \node[scale=\fontscale,anchor=north] at (\xa,\ya) {$0$};
        \node[scale=\fontscale,anchor=north] at (\xb,\ya) {$0.07$};
        \node[scale=\fontscale,anchor=north] at (\xc,\ya) {$0.99\betaK$};
        \node[scale=\fontscale,anchor=north,xshift=-3pt] at (\xe,\ya) {$1/2$};

        \node[scale=\fontscale,anchor=south] at (\xd,\yd) {$0.4$};

        \node[scale=\fontscale,anchor=west] at ($(\xe+0.5,\ya*0.5+\yd*0.5)$) {$\lambda$};
        \node[scale=\fontscale,anchor=west,yshift=3pt] at (\xe,\ya) {$0$};
        \node[scale=\fontscale,anchor=west] at (\xe,\yb) {$2.0$};
        \node[scale=\fontscale,anchor=west] at (\xe,\yc) {$2.149...$};
        \node[scale=\fontscale,anchor=west] at (\xe,\yd) {$2.15$};

        \node[scale=\fontscale,anchor=east] at ($(\xa-0.1,\ya*0.5+\yd*0.5)$) {$Q(\lambda)$};
        \node[scale=\fontscale,anchor=east,yshift=3pt] at (\xa,\ya) {$2$};
        \node[scale=\fontscale,anchor=east] at (\xa,\yc) {$3$};
        \draw[dash pattern=on 6pt off 3pt,line width=4pt,\green] (\xa,\yc) -- (\xc,\yc);
    \end{tikzpicture}
        \caption[LoF entry]{Regions used in gridding for $\ourK=3$ and $\zzz=\zzzsqrt$ to obtain $\JK < 0$ on the bold rectangle, which is the region $(\alpha,c) \in [0.99\betaK,1/2] \times (2,3)$. 
        
        
        
        \parindent=1.5em
        Not to scale. The right edge is labeled with $\lambda \in [0,2.15]$ while the left edge is labeled with $Q(\lambda) \in [2,Q(2.15)]$. The top and bottom edges are both labeled with $\alpha \in [0,0.5]$. 
        }
        \label{fig:zzzsqrt-regions}
    \end{figure}

    \noindent 
    (Region 3) \
    A uniform grid of $40\times 40$ interval rectangles produces 
    \begin{align}
        \mypdv{^2 \Jsqrt3}{\alpha^2}(\alpha,\lambda) \leq -0.01 \foralldisplay{(\alpha,\lambda) \in [0.39,0.5]\times [1.9,2.15]}. \label{eq:nb-3-Jsqrt3-second-partial-alpha}
    \end{align}
    Together with the above exact behavior at $\alpha = \half$ in \Cref{eq:nb-4-Jsqrt3-at-half}, this shows
    \begin{align}
        \label{eq:Jsqrt3-regime-3}
        \Jsqrt3(\alpha,\lambda) \leq \left(\frac{1}{3}-\frac{1}{Q(\lambda)}\right)\ln(2) \foralldisplay{(\alpha,\lambda) \in [0.39,0.5]\times [1.9,2.15]}.
    \end{align}

    \noindent 
    (Region 2a)\ 
    A uniform grid of $200\times 200$ interval rectangles produces 
    \begin{align}
        \label{eq:nb-2a-Jsqrt3-regime}
        \Jsqrt3\left(\alpha,\lambda\right) \leq -10^{-4} \foralldisplay{\left(\alpha,\lambda\right) \in [0.07,0.5]\times [0.0,2.0]}.
    \end{align} (Region 2b)\ 
    A uniform grid of $400\times 400$ interval rectangles produces 
    \begin{align}
        \label{eq:nb-2b-Jsqrt3-regime}
        \Jsqrt3\left(\alpha,\lambda\right) \leq -10^{-4} \foralldisplay{\left(\alpha,\lambda\right) \in [0.07,0.4]\times [2.0,2.15]}.
    \end{align}
    Taking the union of these three result regimes (\Cref{eq:Jsqrt3-regime-3}, \Cref{eq:nb-2a-Jsqrt3-regime}, and \Cref{eq:nb-2b-Jsqrt3-regime}) yields 
    \begin{align}
        \label{eq:almost-done-K3-J-bound}
        \Jsqrt3\left(\alpha,\lambda\right) \leq -\epsilon(c) \foralldisplay{(\alpha,\lambda) \in [0.07,1/2]\times [0,2.15]}.
    \end{align}
    with $c = Q(\lambda)$ and $\epsilon(c) > 0$ for $c<3$, where
    \begin{align*}
        \epsilon(c) = \min\left\{10^{-4},\left(\frac{1}{c}-\frac{1}{3}\right)\ln(2)\right\}.
    \end{align*}
    Recall $c = Q(\lambda)$ introduces the relationships
    \begin{align*}
        c = 2 \iff \lambda = 0, \qquad c=3 \iff \lambda = Q^{-1}(3) \approx 2.149,
    \end{align*}
    so the interval $c \in (2,3]$ corresponds to a subset of $\lambda \in (0,2.15]$.

    For $\ourK=3$, we have $[0.99\betaK,1/2] \subseteq [0.07,1/2]$. Hence the rectangle domain in the theorem statement is covered entirely by \Cref{eq:almost-done-K3-J-bound}, completing the proof.
\end{proof}

\bs

\begin{lemma}
    \label{lem:Jk-upper-near-0} 
    For all $\ourK\geq 3$ and all $c \in (2,\ourK)$,
    \begin{align*}
        \JK(\alpha,\zzzsqrt (\alpha) ;c) \leq \left(\frac{1}{2} - \frac{1}{\ourK}\right) \alpha \ln(\alpha / \betaK) \foralldisplay{\alpha \in (0,\betaK)}.
    \end{align*}
    In particular, for any $\deltaK > 0$, there exists $\epsK > 0$ independent of $c$
    such that 
    \begin{align*}
        \JK(\alpha,\zzzsqrt (\alpha) ;c) < -\epsK \foralldisplay{\alpha \in [\deltaK,0.99\betaK], c \in (2,\ourK)}.
    \end{align*}    
\end{lemma}
\begin{proof}
    
    This is an analog to the first half of \cite{pittel2016satisfiability} Claim 5.1. We handle the details of our proof as follows.
    
    With $\lambda = Q^{-1}(c)$, for $x\geq 0$,  Equation 5.3 of
    \cite{pittel2016satisfiability} 
    with $\littlek=1$ proves
    \begin{align*}
        \frac{\expMTwo(\lambda x)}{\expMTwo(\lambda)} \leq \exp((x-1)c).
    \end{align*}
    Next, they 
    apply this inequality 
    to get 
    \begin{align*}
        &\frac{1}{c} \ln \frac{\expMTwo(\lambda \cdot (\zeta_2 + \zeta_1)) + \expMTwo(\lambda \cdot (\zeta_2 - \zeta_1))}{2 \expMTwo(\lambda)}
        \ \leq\  (\zeta_2-1) + c\,\zeta_1^2/2.
    \end{align*}
    provided
    $\zeta_2 + \zeta_1 \geq 0$ and $\zeta_2 - \zeta_1 \geq 0$.
    Substituting this inequality into \Cref{def:Jkm} yields
    \begin{align*}
        &\JK(\alpha,\zzz (\alpha) ;c) \leq \frac{1}{\ourK} H(\alpha) + \alpha \ln(\alpha/\zeta_1) + \ov\alpha \ln(\ov\alpha / \zeta_2) + (\zeta_2-1) + c\,\zeta_1^2/2.
    \end{align*}
    Now we use $H(\alpha) \leq -\alpha\ln(\alpha) + \alpha$, and substituting the definition of $\zzzsqrt$ from \Cref{def:eq-zzzsqrt} yields
    \begin{align*}
        \JK(\alpha,\zzzsqrt (\alpha);c) &\leq \frac{1}{\ourK} (\alpha - \alpha\ln(\alpha)) + \alpha \ln(\sqrt{(\ourK-1)\alpha}) - \alpha + c\alpha/(2(\ourK-1))
        \\ &= \left(\frac{1}{2} - \frac{1}{\ourK}\right) \alpha \ln(\alpha) + \alpha \left(\frac{1}{\ourK} + \ln(\sqrt{\ourK-1}) - 1 + \frac{c}{2(\ourK-1)}\right)
    \end{align*}
    Using $c\leq \ourK$ and following the definition of $\betaK$ in \Cref{def:beta-k}, we have
    \begin{align}
        \label{eq:Jk-sqrt-upper-bound-near-0}
        \JK(\alpha,\zzzsqrt (\alpha);c) &\leq \left(\frac{1}{2} - \frac{1}{\ourK}\right) \alpha \ln(\alpha / \betaK).
    \end{align}
    The second part of the claim, involving $\epsK$ and $\deltaK$, follows immediately since the upper bound in \Cref{eq:Jk-sqrt-upper-bound-near-0} is continuous and negative for $\alpha \in (0, \betaK)$.
\end{proof}

\subsection{Bounds along a piecewise choice for $\zzz$.}

This section introduces a function $\zzzComp$ and bounds $\JK(\alpha,\zzzComp(\alpha);c)$ for all $\alpha \in (0,1]$. The function $\zzzComp$ is defined piecewise on $\alpha \in [0,1]$ with the pieces being intervals determined by $\deltaBallK(c)$, which we introduce now.
There exists positive functions $\deltaBallK$ and $\epsBallK$ continuous on $c \in (2,\ourK)$ such that for all $c\in (2,\ourK)$, we have $\deltaBallK(c) < \betaK$ and
\begin{align*}
    \JK(\alpha,(1-\deltaBallK(c),\deltaBallK(c));c) \leq -\epsBallK(c) \foralldisplay{\alpha \in [1-\deltaBallK(c), 1]}.
\end{align*}
This is asserted in equation (5.3) of \cite{pittel2016satisfiability}, though we describe a full proof with detailed hypotheses in the appendix, \Cref{lem:deltaBallK}.

For $\ourK \geq 3$, define $\zzzComp$ in terms of $\zzzHelper = (\zetaHelper_1, \zetaHelper_2)$ as follows
\begin{align}
    \label{eq:monster-zzz}
    \mathdf{\zzzComp(\alpha)} &= \begin{cases}
        \zzzHelper(\alpha), & \alpha \in [0,\half],
        \\ (\zetaHelper_2(1-\alpha),\zetaHelper_1(1-\alpha)), &\alpha \in (\half, 1-\deltaBallK(c)],
        \\ (1-\deltaBallK(c), \deltaBallK(c)), &\alpha \in (1-\deltaBallK(c),1],
    \end{cases}
\end{align}
where
\begin{align*}
    \mathdf{\zzzHelper(\alpha)} &= \begin{cases}
        \zzzsqrt(\alpha), & \alpha \in [0, \; 0.99\betaK],
        \\ \zzzsqrt(\alpha), & \alpha \in [0.99\betaK, \half] \text{ if } K=3,
        \\ \zzzlin(\alpha), & \alpha \in [0.99\betaK, \half] \text{ if } K\geq 4.
    \end{cases}
\end{align*}
We emphasize $\zzzsqrt(\alpha)$, defined in \Cref{def:eq-zzzsqrt}, depends on $\ourK$.

\begin{lemma}
    \label{prop:Jk-upper-start}
    For all $\ourK\geq 3$ and $\deltaK \in (0,\half)$, there exist a positive function $\epsK$ continuous in $c \in (2,\ourK)$ such that
    \begin{align}
        \label{eq:JK-epsilon-final}
        \JK(\alpha,\zzzHelper(\alpha);c) \leq -\epsK(c) \foralldisplay{\alpha \in [\deltaK,1/2],\, c\in (2,\ourK)}.
    \end{align}
\end{lemma}

\begin{proof}
    This is an analog of the second half of Claim 5.1 and all of Claim 5.2 of \cite{pittel2016satisfiability}.

    We handle the claimed domain ($K\geq 3$ and $\alpha \in [\deltaK,\half]$) in three cases.
    \begin{enumerate}
        \item Over $\alpha \in [\deltaK,0.99\beta_k]$, for all $K \geq 3$, we have $\zzzHelper = \zzzsqrt$, so \Cref{lem:Jk-upper-near-0} implies \Cref{eq:JK-epsilon-final} for some $\epsK$ which is constant (hence continuous) with respect to $c$.
        
        \item Over $\alpha \in [0.99\betaK,\half]$, for $K = 3$, we have $\zzzHelper = \zzzsqrt$, so \Cref{lem:nb-K3-Jsqrt-bound} implies \Cref{eq:JK-epsilon-final} for the function $\epsK \coloneqq \epsilon$, which is continuous with respect to $c$.
        
        \item Over $\alpha \in [0.99\beta_k,\half]$, for $K \geq 4$, we have $\zzzHelper = \zzzlin$, so \Cref{lem:Kgeq4-J-bound-allc} implies \Cref{eq:JK-epsilon-final} for some function $\epsK$ which is continuous with respect to $c$.
    \end{enumerate}
    Taking the minimum of the $\epsK$ from each case gives the $\epsK$ for the lemma.
\end{proof}

\begin{lemma}
    \label{lem:Jk-upper-near-1}
    
    For all $\ourK\geq 3$, there exists a positive function $\epsK$ continuous in $c \in (2,\ourK)$ such that
    \begin{align*}
        \JK(\alpha,\zzzComp(\alpha);c) < -\epsK(c) \foralldisplay{\alpha \in [1/2,1],\, c\in (2,\ourK)}.
    \end{align*}
\end{lemma}

\begin{proof} 
    This is an analog of Claim 5.3 of \cite{pittel2016satisfiability}.
    We handle the claimed domain, $\alpha \in [1/2, 1]$, in two cases.

  Over $\alpha \in [1-\deltaBallK(c), 1]$,
    the choice of the continuous functions $\deltaBallK$ and $\epsBallK$ in \Cref{lem:deltaBallK} ensures that 
    \begin{align*}
        \JK(\alpha,\zzzComp(\alpha);c) &= \JK(\alpha,(1-\deltaBallK(c),\deltaBallK(c));c) \leq -\epsBallK(c).
    \end{align*}
    
    
    Note
    the definition of $\JK$ in \Cref{def:Jkm} yields
    \begin{align}
        & \JK(1-\alpha,(\zeta_1,\zeta_2);c) - \JK(\alpha,(\zeta_2,\zeta_1),c) \label{eq:JK-difference}
        \\ &= \frac{1}{c}\ln \frac{\expMTwo(x) + \expMTwo(y)}{2 \expMTwo(\lambda)} - \frac{1}{c} \ln \frac{\expMTwo(x) + \expMTwo(-y)}{2 \expMTwo(\lambda)} \nonumber
    \end{align}
    where $x = \lambda \cdot (\zeta_1 + \zeta_2)$ and $y = \lambda \cdot (\zeta_2 - \zeta_1)$. For $\alpha \in [1/2,1-\deltaBallK(c)]$, taking $\zeta_2 = \zetaHelper_2(1-\alpha)$ and $\zeta_1 = \zetaHelper_1(1-\alpha)$ gives $y \geq 0$ since $\lambda > 0$, hence $\expMTwo(y) \geq \expMTwo(-y)$. Thus the difference in \Cref{eq:JK-difference} is non-negative, and we find
    \begin{align}
     \label{eq:ineq-by-JK-limited-asymmetry}
        \JK(\alpha,\zzzComp(\alpha);c) &= \JK(\alpha,(\zetaHelper_2(1-\alpha),\zetaHelper_1(1-\alpha));c)
        \\ &\leq \JK(1-\alpha,(\zetaHelper_1(1-\alpha),\zetaHelper_2(1-\alpha));c)
        = \JK(1-\alpha,\zzzHelper(1-\alpha);c). \nonumber
    \end{align}
    Since $\alpha \in [1/2,1-\deltaBallK(c)]$, we have $1-\alpha \in [\deltaBallK(c),1/2]$, so \Cref{prop:Jk-upper-start} shows there exists $\expandafter\tilde\epsK(c) > 0$ continuous in $c$ such that
    \begin{align*}
        \JK(\alpha,\zzzComp(\alpha);c) \leq \JK(1-\alpha,(\zetaHelper_1(1-\alpha),\zetaHelper_2(1-\alpha));c) \leq -\expandafter\tilde\epsK(c).
    \end{align*}
    
    Taking the minimum of $\expandafter\tilde\epsK(c)$ and $\epsBallK(c)$ gives the desired $\epsK(c)$.
\end{proof}

\bs

\begin{proposition}
    \label{prop:Jk-upper}
    For all $\ourK\geq 3$, there exists positive functions $\epsK$ and $\zetaLowerK$ continuous in $c \in (2,\ourK)$, and a function $\zzzKc(\alpha) = (\zeta_1(\alpha),\zeta_2(\alpha))$ satisfying $\zeta_2(\alpha) \geq \zetaLowerK(c)$ such that for all $c \in (2,\ourK)$,
    \begin{align}
        \JK(\alpha,\zzzKc(\alpha);c) &\leq \left(\frac{1}{2} - \frac{1}{\ourK}\right) \alpha \ln(\alpha / \betaK) &&\textforall{\alpha \in (0, 0.99\betaK)},
        \label{eq:lem-Jk-upper-near-0}
        \\ \JK(\alpha,\zzzKc(\alpha);c) &\leq -\epsK(c) &&\textforall{\alpha \in [\betaK/2,\,1]}.
        \label{eq:lem-Jk-upper-onwards}
    \end{align}
\end{proposition}

\begin{proof}
    For each $\ourK\geq 3$ and $c \in (2,\ourK)$, the function $\zzzKc$ is given by $\zzzKc = \zzzComp$, as defined in \Cref{eq:monster-zzz}. Determine $\zetaLowerK$ by
    \begin{align*}
        \mathdf{\zetaLowerK(c)} = \max\left\{\deltaBallK(c), \sqrt{\deltaBallK(c)/(K-1)}\right\}.
    \end{align*}
    Inspection shows $\zeta_2(\alpha) \geq \zetaLowerK$ in each piece of the definition of $\zzzComp$.
    
    Over $\alpha \in (0,0.99\betaK)$, we have $\zzzComp(\alpha) = \zzzsqrt(\alpha)$, so \Cref{lem:Jk-upper-near-0} proves \Cref{eq:lem-Jk-upper-near-0}.

    To prove \Cref{eq:lem-Jk-upper-onwards},
    we handle its domain in two cases.
    \Cref{prop:Jk-upper-start} with $\deltaK = \betaK/2$ proves \Cref{eq:lem-Jk-upper-onwards} over $\alpha \in [\betaK/2,\,\half]$, while \Cref{lem:Jk-upper-near-1} proves it over $\alpha \in [\half,1]$. Each lemma provides a positive function $\epsK$ continuous in $c$; taking the minimum of these yields the $\epsK$ for the lemma.
\end{proof}

This proposition is the underpinning of the next section.

\section{Upper bound on expected number of critical row sets in terms of $m$}
\label{sec:boundOnExpect}

As we recall from \Cref{sec:random-variables-YZ}, the random variable $\mathdf{\Zmnl} = X^{(\ell)}(\Gamma)$ is the number of non-empty critical row subsets of $\Gamma$ of cardinality $\ell$, when $\Gamma$ is generated uniformly randomly on the space of 2-core \KXORGAME{} equations of the size $(m,\bn)$.
This section proves an upper bound on the expectation of $X(\Gamma) = \sum_{\ell=2}^m \Zmnl$.
The upper bound is tight enough that it fulfills our
main wish: when $\lim m/n_i < \cThreshK$ and $m,n_i \to \infty$,
we have $\Ex[\sum_{\ell=2}^m \Zmnl] \to 0$.
The upper bound 
follows directly from the estimates on $J_K$ given in
\Cref{prop:Jk-upper}, and we now state it.

This is the main component of proving \Cref{thm:main-theorem-core-sat}.

\begin{theorem} 
    \label{thm:sum-upper-bound-O}
    Fix $\ourK\geq 3$. As $m,n_i \to \infty$ with $\lim m/n_i \in (2, \ourK)$ for $i\in\{1,2,\ldots,\ourK\}$, we have 
    \begin{align*}
     \sum_{\ell = 2}^m \Ex[\Zmnl] = O\left(m^{-2\ourK\left(\half - \frac{1}{\ourK}\right)}\right) = O\left(m^{2-\ourK}\right).
    \end{align*}
\end{theorem}
\begin{proof}
    (This is the analog of Corollary 4.3 from \cite{pittel2016satisfiability}.)

    To bound each term in the sum, we will pick the $\zzzKci$ whose existence and properties are stated in \Cref{prop:Jk-upper} with 
    $c_i \coloneqq m/n_i$. \Cref{prop:Jk-upper} involves functions $\zetaLowerK(c)$ and $\epsK(c)$. It will be useful to show these are uniformly bounded away from zero for all but finitely many $(m,\bn)$.
    
    There exists a single compact interval $I \subset (2,\ourK)$ such that $c_i \coloneqq m/n_i \in I$ for each $i$, for sufficiently large $m$. This is because $\lim m/n_i \in (2,\ourK)$ holds for each $i$. Since $I$ is compact, and $\zetaLowerK$ and $\epsK$ are each positive and continuous in $c$, we can define the constants $\zetaLLK$ and $\epsLLK$ by 
    \begin{align*}
        \zetaLLK = \min\{\zetaLowerK(c) \mid c\in I\} > 0, \qquad \epsLLK = \min\{\epsK(c) \mid c\in I\} > 0.
    \end{align*}
    For all but finitely many $(m,\bn)$, we have $c_i \coloneqq m/n_i \in I$ and
    \begin{align}
        \zetaLowerK(m/n_i) \geq \zetaLLK > 0, \qquad \epsK(m/n_i) \geq \epsLLK > 0. \label{eq:zeta-eps-Omega-1}
    \end{align}
    From \Cref{prop:ZbdJ},
    \begin{align*}
        \Ex[\Zmnl] &\leq O(1) \frac{(\ell)^{(\ourK-1)/2}}{\sqrt{(\zeta_2)^\ourK}} \exp(m \sum_{i=1}^\ourK \JK(\alpha,\zzzKci;c_i)).
    \end{align*}
    \Cref{prop:Jk-upper} gives $\zeta_2 \geq \zetaLowerK$, so $\zeta_2 \geq \zetaLLK$, and we collect the denominator into the $O(1)$ to find
    \begin{align*}
        \Ex[\Zmnl] &\leq O(1)(\ell)^{(\ourK-1)/2} \exp(m \sum_{i=1}^\ourK \JK(\alpha,\zzzKci;c_i)).
    \end{align*}
    Let $\ell_\ourK = \betaK m$. 
    We are interested in the sum 
    $\sum_{\ell=2}^m  \Ex[\Zmnl]$, which we split into two parts as
    \begin{align*}
       \ell \in   [2, m] = [2,\ell_\ourK/2] \ \cup \ (\ell_\ourK/2, m]. 
    \end{align*}
    {\bs\noindent\emph{First part of sum}} ($\ell \in [2,\ell_\ourK/2]$):
    Here $\alpha = \ell/m \in (0,\betaK/2] \subseteq (0,0.99\betaK]$. We begin by bounding the summand. \Cref{prop:Jk-upper} gives
    \begin{align*}
        \Ex[\Zmnl] &\leq O(1)(\ell)^{(\ourK-1)/2} \exp(m \sum_{i=1}^\ourK \left(\half-\frac{1}{\ourK}\right) \alpha \ln(\alpha/\betaK)).
        \\ &\leq O(1)(\ell)^{(\ourK-1)/2} \exp(\ourK \left(\half-\frac{1}{\ourK}\right) \ell \ln(\ell/\ell_\ourK)),
    \end{align*}
    By convexity of $\ell \ln(\ell / \ell_\ourK)$ in $\ell$, interpolating on $\ell \in [2,\ell_\ourK/2]$ yields
    \begin{align*}
        \ell \ln (\ell/\ell_\ourK) 
        &\leq 2\ln(2/\ell_\ourK) + \frac{\ell - 2}{\ell_\ourK/2 - 2} \left((\ell_\ourK/2)\ln(1/2) - 2\ln(2/\ell_\ourK)\right)
        \\ &\leq 2\ln(2/\ell_\ourK) + (\ell-2)(-\ln2 + o(1))
        \\ &\leq 2\ln(2/\ell_\ourK) - 0.6(\ell-2)
    \end{align*}
    for sufficiently large $m$, where we used $\ell_\ourK \to \infty$ as $m\to\infty$, and $0.6 < \ln2$. Thus
    \begin{align*}
        \Ex[\Zmnl] &\leq O(1)(\ell)^{(\ourK-1)/2} \exp(\ourK \left(\half-\frac{1}{\ourK}\right) [2\ln(2/\ell_\ourK)-0.6(\ell-2)])
        \\ &\leq O(1)(\ell)^{(\ourK-1)/2} m^{-2\ourK\left(\half - \frac{1}{\ourK}\right)} \exp(-0.6 \ourK \left(\half-\frac{1}{\ourK}\right) (\ell-2)).
    \end{align*}
    where the last line incorporates $ 
     (2m /\ell_\ourK)^{2\ourK(\half-\frac{1}{\ourK})}=
    (2/\betaK)^{2\ourK(\half-\frac{1}{\ourK})} = O(1)$. 
    
    With bounds complete, we return to the first part of the sum. For brevity, let
    $
        r \coloneqq \exp(-0.6 \ourK \left(\half-\frac{1}{\ourK}\right)),
    $
    so
    \begin{align}
        \sum_{\ell=2}^{\floor{\ell_\ourK/2}} 
        \Ex[\Zmnl]
        \leq O(1) m^{-2\ourK\left(\half - \frac{1}{\ourK}\right)} \sum_{\ell=2}^\infty  
        (\ell)^{(\ourK-1)/2} \, r^\ell 
      \label{eq:summation-rhs}
    \end{align}
    The summation in \Cref{eq:summation-rhs} depends only on $\ourK$ and converges since $r < 1$. Hence as $m,n_i\to\infty$,
    \begin{align}
        \label{eq:partial-sum-Zmnl-1-bound}
        \sum_{\ell=2}^{\floor{\ell_\ourK/2}} \Ex[\Zmnl] = O\left(m^{-2\ourK\left(\half - \frac{1}{\ourK}\right)}\right).
    \end{align}
    
    {\bs\noindent\emph{Second part of sum}} ($\ell \in (\ell_\ourK/2,m]$). Here, $\alpha = \ell/m \in (\betaK/2,1]$. By \Cref{prop:Jk-upper}, we have $\Ex[\Zmnl] \leq O(1) \exp(-\epsK m)$, giving
    \begin{align*}
        \sum_{\ell=\floor{\ell_\ourK/2}+1}^m \Ex[\Zmnl] = O(m) \exp(-\epsK m).
    \end{align*}
    From \Cref{eq:zeta-eps-Omega-1}, we have $\epsK \geq \epsLLK$, so
    \begin{align}
        \label{eq:partial-sum-Zmnl-2-bound}
        \sum_{\ell=\floor{\ell_\ourK/2}+1}^m \Ex[\Zmnl] = O(m e^{-\epsLLK m}).
    \end{align}

    {\bs\noindent\emph{Finishing:}} Adding the two partial sums (\Cref{eq:partial-sum-Zmnl-2-bound} and \Cref{eq:partial-sum-Zmnl-1-bound}) finishes this proof of \Cref{thm:sum-upper-bound-O}.
\end{proof}

\section{Translating from 2-cores of a graph  to 
general graphs}
\label{sec:translating-2-cores}

This section provides the bridge which takes us from the (hard-won) theorem that the satisfiability threshold for a 2-core 
\KXORGAME{} is $\ourK$, to the main theorem of this paper,
\Cref{thm:main-thm-satisfiability}, that the satisfiability threshold for unconstrained \KXORGAME{} is $\cThreshK$.
The bridge has two main spans. First, in \Cref{sec:bwz-core-reduction}, we use \cite{botelho2012cores}, which predicts 
the size of the 2-core of a uniformly random \KXORGAME{} matrix. 
Next, in \Cref{sec:maintenance-of-uniformity}, we prove in our situation what is called the
``maintenance of uniformity;" when a randomly-generated
\KXORGAME{} graph $G$ is reduced to its 2-core, the 2-core $\fin G$
is uniformly distributed conditioned on the size of the 2-core.
\Cref{sec:proofFinale} attaches our bridge to the rest of the paper to prove \Cref{thm:main-thm-satisfiability}.

As mentioned in the introduction, there
are two ways to view and state all of the results of this paper: one is with matrices, and another is with hypergraphs.
The predominate perspective we used so far is with matrices.
This section relies heavily on \cite{botelho2012cores},
so we introduce the hypergraph language here in \Cref{sec:matrix-hypergraph-interp},
with reassurance to the reader that if they prefer matrices
the conversion should be easy for them.

\subsection{Hypergraph interpretation}
\label{sec:matrix-hypergraph-interp}

For a boolean matrix $\Gamma \in \z2^{m\times n}$, we construct a hyper-graph as follows. Let the vertices $V = \{1,\ldots,n\}$ index the columns of $\Gamma$, and determine the edges by regarding each row of $\Gamma$ as a hyper-edge containing all the vertices indexed with a 1 in the row. Thus there are $|V|=n$ vertices and $|E|=m$ edges.

With this interpretation, a \kXORSAT{} matrix corresponds to a $\littlek$-uniform hypergraph, meaning each edge contains exactly $\littlek$ vertices of $V$.

With the same interpretation, a \KXORGAME{} matrix corresponds to a $\ourK$-uniform $\ourK$-partite hypergraph, where the vertices $V$ are partitioned into $\ourK$ sets $V_1,\ldots,V_\ourK$, and every edge contains exactly one vertex in each part $V_i$ of the partition. Recall the size of the \KXORGAME{} is divided into blocks, described by the vector $\bn=(n_1,\ldots,n_\ourK)$, where block $j$ has $n_j$ columns. This corresponds to each set $V_j$ containing exactly $n_j$ vertices.

Note the interpretation technically produces multi-hypergraphs, since a duplicated row produces a duplicated edge. The distinction does not matter for the purpose of this paper, since there are asymptotically almost surely no duplicated rows in a uniformly random \kXORSAT{} or \KXORGAME{} matrix, as $m,n\to\infty$ with $m = O(n)$, if $\littlek \geq 3$ (resp., as $m,\bn\to\infty$ with $m = O(n_i)$, if $\ourK \geq 3$).

In \Cref{sec:def-3-game-and-2-cores}, \kXORSAT{} and \KXORGAME{} problems were defined by randomly generating a matrix $\Gamma$ and vector $s$, where the problem is satisfiable if $\Gamma x = s$ has a solution. The analog with hyper-graphs is randomly generating a hyper-graph $(V,E)$ along with a random assignment of 1 of 2 colors, say green or orange, to each hyper-edge.
The problem is satisfiable if the vertices of the hyper-graph can be colored red and blue such that each green hyper-edge contains an odd number of red vertices, and each orange hyper-edge contains an even number of red vertices.

The 2-core notion stated for matrices in \Cref{sec:def-3-game-and-2-cores} corresponds to the standard notion of 2-cores of a hypergraph, since submatrices correspond to subgraphs. The 2-core $\fin G$ of a hypergraph $G$ is the largest induced sub-hypergraph of minimum degree at least 2. Note this notion of sub-hypergraph does not allow for partial edges; an edge is included in $E(\fin G)$ precisely when all vertices of the edge are included in $V(\fin G)$. In this section, we let $\fin m$ be the number of edges in $\fin G$, and $\fin\bn = (\fin n_1, \ldots, n_\ourK)$ describe the number of vertices in each part of the partitioned vertices, so $n_1 = |\fin V_1|, \ldots, n_\ourK = |\fin V_\ourK|$.

\subsection{Satisfiability threshold definitions}

Recall from 
\Cref{sec:mainThms} the definitions 
$ Q(z) = \frac{z(e^z-1)}{e^z-1-z},$ 
and 
\begin{align}
   \mathdf{h_{\ourK}(\mu)} &\coloneqq \frac{\mu}{(e^{-\mu} (e^\mu - 1))^{\ourK - 1}}, & 
    \mathdf{\cThreshK} &\coloneqq h_\ourK(Q^{-1}(\ourK)). \label{eq:def:cThreshK}
\end{align}
Define the continuous increasing function $F$ by $\mathdf{F(z)} = 1 + \frac{e^z - 1}{z}$.
For $\ourK \geq 3$, let $\muTurnK = F^{-1}(\ourK)$. By calculus (cf. Section 1.1 of \cite{botelho2012cores}),  
$h_\ourK$ is decreasing for $0 < \mu < \muTurnK$ and increasing for $\mu > \muTurnK$. For $\ourK \geq 3$, let
\begin{align*}
    \mathdf{\cTurnK} &\coloneqq h_\ourK(\muTurnK) = \inf \{h_{\ourK}(\mu) \mid \mu > 0\}.
\end{align*}
For $\ourK \geq 3$ and $c > \cTurnK$, let $\mathdf{\muK(c)}$ be the larger of the two solutions to $h_\ourK(\mu) = c$, that is,
\begin{align}
    \mathdf{\muK(c)} \coloneqq \max \{\mu \mid h_\ourK(\mu) = c\}. \label{eq:def-mu-k}
\end{align}
Note $\muK\colon (\cTurnK,\infty) \to (\muTurnK,\infty)$ is an increasing function and is a right inverse to $h_\ourK$. In particular, $\mu_\ourK$ is a left inverse to $h_\ourK$ restricted to $(\cTurnK,\infty)$.

\begin{lemma}
    For each $K\geq 3$, we have
    \begin{align*}
        Q^{-1}(\ourK) > \muTurnK, \quad \mathrm{and}\quad \cTurnK < \cThreshK.
    \end{align*}
\end{lemma}

\begin{proof}
    Calculate $F'(x) = \frac{e^x}{x} - \frac{e^x-1}{x^2}$,
    hence
    \begin{align*}
        x^2 (F'(x) - 1) &= x e^x - (e^x - 1) - x^2
        = (x-1)(e^x - 1 - x).
    \end{align*}
    Thus for $x \geq 1$, we have $F'(x) \geq 1$.
    Theorem 3.2 of \cite{stirlingConnDobro} provides $\frac{1}{3} < Q'(x) < 1$ for all $x>0$, hence $F'(x) > Q'(x)$ for $x \geq 1$. Direct evaluation yields 
    \begin{align*}
        e = F(1) > Q(1) \approx 2.39,
    \end{align*}
    so $F(x) > Q(x)$ for all $x \geq 1$, and $F^{-1}(y) < Q^{-1}(y)$ for all $y > Q(1) \approx 2.39$. Note $K \geq 3$, so
    $
        \muTurnK = F^{-1}(\ourK) < Q^{-1}(\ourK).
    $
    Hence $Q^{-1}(\ourK) > \muTurnK$ for each $\ourK \geq 3$, so 
    \begin{equation*}
        \cTurnK = h_\ourK(\muTurnK) < h_\ourK(Q^{-1}(\ourK)) = \cThreshK. \qedhere
    \end{equation*}  
\end{proof}

\begin{lemma}
    \label{lem:Q-beats-K-after-thresh}
    If $c > \cThreshK$, then $Q(\muK(c)) > \ourK$.
\end{lemma}
\begin{proof}
    Since $Q^{-1}(\ourK) \geq \muTurnK$, 
    we have $\muK$ is in fact a left inverse to $h_\ourK$ applied at $Q^{-1}(\ourK)$. Applying $\muK$ to both sides of \Cref{eq:def:cThreshK} yields
    \begin{align*}
        \muK(\cThreshK) = Q^{-1}(\ourK)
        \qquad  \mathrm{equivalently} \qquad Q(\muK(\cThreshK)) = \ourK.
    \end{align*}
    The functions $Q$ and $\muK$ are increasing, so $Q(\muK(c))$ is increasing in $c$.
\end{proof}

\subsection{Predicted size of the 2-core}
\label{sec:bwz-core-reduction}

\KXORGAME{} matrices correspond to $\ourK$-uniform $\ourK$-partite hypergraphs $G$ as discussed in \Cref{sec:matrix-hypergraph-interp}. The 2-core of these hypergraphs are relevant to perfect hash functions, cf. \cite{botelhoHashFunctions}, so \cite{botelho2012cores} has analyzed their expected sizes.
Their main theorem produces (a slight modification of) the following lemma, which 
we shall apply to our situation.

\begin{lemma}
    \label{lem:bwz-core-reduction}
    Let $c>0$ and integer $\ourK\geq 3$ be fixed. 
    Suppose $m,n\to\infty$ with $\lim m/n = c$, and random hypergraphs  $G$ are uniformly distributed on the space of $\ourK$-partite $\ourK$-uniform hypergraphs on $\ourK n$ vertices with $m$ edges and each part containing $n$ vertices.
    These correspond to \KXORGAME{} matrices as described in \Cref{sec:matrix-hypergraph-interp}.

    For $c < \cTurnK$, $G$ has an empty 2-core asymptotically almost surely (\aas{}).

    For $c > \cTurnK$, the 2-core of $G$ a.a.s has $\fin n$ vertices and $\fin m$ hyperedges, with
    \begin{align*}
        \fin n &= e^{-\mu} (e^\mu - 1 - \mu) \ourK n (1 + o(1))
    \\  \fin m &= \mu e^{-\mu} (e^\mu - 1) n (1 + o(1)),
    \end{align*}
    where $\mathdf{\mu} \coloneqq \muK(c)$ is defined in \Cref{eq:def-mu-k}.
    Furthermore, the 2-core of $G$ a.a.s has $\fin n_j$ vertices in the $j$th part ($1 \leq j \leq \ourK)$, with
    \begin{align}
    \label{eq:ABCuniformity}
        \fin n_j &= e^{-\mu} (e^\mu - 1 - \mu) n (1 + o(1)).
    \end{align}
    Hence as $m,n \to \infty$, then a.a.s, $\fin m, \fin n_j \to \infty$ and
    \begin{align*}
        \lim \frac{\fin m}{\fin n_j} 
        = \frac{ e^{-\mu} \mu (e^\mu - 1)}{e^{-\mu} (e^\mu - 1 - \mu)} 
        = Q(\mu).
    \end{align*}
    We emphasize that this limit is the same for all $j$.
\end{lemma}

\begin{proof}
    The first three equations follow from Theorem 3 of \cite{botelho2012cores}. 
    Our $\ourK$ is their $r$, and we set their $\littlek$ to 2 (to restrict to 2-cores).

    However, the claim about $\fin n_j$ in \Cref{eq:ABCuniformity} is not stated directly in their Theorem 3, but it is apparent from the proof
    as we now discuss. 
    Based on the concentration inequalities in \cite{wormald1995diffeq} (which come from Azuma's inequality), the size of each block \aas{} is closely approximated by
    the solutions of the key differential equations (Equations (14)--(15) of \cite{botelho2012cores}). The differential equation for each block is identical, so the solutions are the same, and $\fin n = \ourK \fin n_j$.

    The limit claim follows from standard algebra.
\end{proof}

For historical perspective, recall 
\kXORSAT{} deals with $\littlek$-uniform hypergraphs which are not necessarily $\littlek$-partite. The analogous theorem to \Cref{lem:bwz-core-reduction} for this simpler case is treated in \cite{molloy2003cores}.

\subsection{Maintenance of uniformity}
\label{sec:maintenance-of-uniformity}

We provide a concise proof for maintenance of uniform distribution from an initial random $\ourK$-partite hypergraph to its 2-core. Note previous work such as \cite{broder1993} Lemma 4.1, \cite{botelho2012cores} Observation 1, and \cite{dubois20023} Lemma 4.1, break up the 2-core reduction into an algorithm that can take many steps, and they show each step preserves a uniform distribution. 
Our proof of the lemma is a bit different from these earlier proofs because we shows the uniform distribution of the 2-core without passing through a sequential algorithm.

Let $\mathdf{\Ximn}$ denote the space of \KXORGAME{} hypergraphs on $\ourK n$ vertices with $m$ edges. Also, let $\Psimnfin$ be the space of 2-core \KXORGAME{} hypergraphs with $\ourK$ parts, $\fin n_j$ vertices in the $j$th part, and $\fin m$ edges. Recall $\Psimnfin$ was defined earlier in \Cref{sec:def-Psimn} as a set of matrices, but this abuse of notation is justified by the discussion in \Cref{sec:matrix-hypergraph-interp}.

\begin{lemma}
    \label{lem:maintain-uniformity-general}
    Let $H$ be some hypergraph, and $\fin G = (\fin V, \fin E)$ a sub-hypergraph of $H$ with $\fin m = |\fin E|$ edges. 
    Let $\fin V = V(\fin G)$, 
    and let $m$ be an integer. 
    
    Then the number of hypergraphs $G \subseteq H$ such that $\fin G$ is the 2-core of $G$
    and $G$ has $m$ edges, depends only on $m$, $\fin m$, $H$, and $\fin V$, and not on the full edge set $\fin E$.
\end{lemma}
\begin{proof}
    Let $W \subseteq E(H)$ be the set of all edges which contain at least one vertex not in $\fin V$.
    The set of graphs $G=(V,E)\in \Ximn$ with $\fin G$ as its 2-core is in bijection with the set of $F \subseteq W$ with $\abs{F} = m - \fin m$ such that $(V,\fin E \cup F)$ has $\fin G$ as its 2-core. We prove this bijection next.

    Each graph $G=(V,E)\in \Ximn$ with $E\supseteq \fin E$ can be represented uniquely as the union of the sub-graph $J$ containing only edges in $E(H) \setminus W$, and the sub-graph $F$ containing only edges in $W$. However, if $\fin G = (\fin V, \fin E)$ is the 2-core of $G$, we will always have $J = \fin E$, since any edge $e \in E(H) \setminus W$ consists only of vertices in $\fin V$, so adding $e$ to the edge set would produce a larger 2-core, contrary to a 2-core's maximality property. This establishes the bijection.

    The bijection finishes the proof of the lemma, as we now show that the set of such $F$ depends only $m,\fin m, H,$ and $\fin V$.
    Since $\fin G$ is a sub-graph of $(V,\fin E \cup F)$, automatically the 2-core of $(V, \fin E \cup F)$ is a super-graph of $\fin G$. If the 2-core of $(V,\fin E \cup F)$ were not $\fin G$, then there would be an edge set $E' = \fin E \cup F'$ with $F' \subseteq F$ such that $E'$ contains each vertex in $V$ in either 0 or at least 2 edges. However,
    \begin{itemize}
        \item If $v \in \fin V$, then $v$ is contained in at least 2 edges of $\fin E$ and hence in at least 2 edges of $E'$.
        \item If $v \notin \fin V$, then $v$ is contained in no edges of $\fin E$, so $v$ is contained in the same number of edges in $E'$ as it is in $F' \subseteq F$.
    \end{itemize}
    Thus the existence of such an edge set $E'$ depends only on $F$ and $\fin V$.
\end{proof}

Maintenance of uniformity now follows. 
\begin{proposition}
    \label{lem:maintain-uniformity-1}
    Suppose $G$ is uniformly distributed on the space $\Ximn$ of \KXORGAME{} hypergraphs on $\ourK n$ vertices with $m$ edges, and $\fin G$ is its 2-core. Then conditioned on the size $(\fin m, \fin \bn)$ of $\fin G$, we have $\fin G$ is equally likely to be any 2-core \KXORGAME{} hypergraph of that size, $\Psimnfin$.
    
\end{proposition}
\begin{proof}    

    We only need to show that the number of hypergraphs $G \in \Ximn$ which have $\fin G$ as its 2-core depends only on $m,\bn,\fin m, \fin \bn$ and not on the particular $\fin G$. This follows directly from \Cref{lem:maintain-uniformity-general}, by fixing $V$ and $\fin V$ based on $\bn$ and $\fin \bn$ without loss of generality due to the symmetry of $\Ximn$ and $\Psimn$, and letting $H$ be the complete hypergraph containing every possible edge in $\Ximn$.
\end{proof}

\subsection{Proof of \Cref{thm:main-thm-satisfiability}}
\label{sec:proofFinale}

Now we return  to matrix terminology.

\begin{proof}
    Suppose $\ourK \geq 3$ and $c > 2$. Let $m,n \to\infty$ with $\lim m/n = c$. Generate $(\Gamma,s)$ uniformly at random on the space of uniformly-tiled $\ourK$-XORGAME equations of this size ($m$ rows, $\ourK$ blocks of $n$ columns each).

    For each $\Gamma$, we analyze its 2-core $\fin\Gamma$. Let $\fin n_j$ be the number of vertices in block $j$ of $\fin\Gamma$, and $\fin m$ be the number of edges in $\fin\Gamma$. As discussed in \Cref{sec:def-3-game-and-2-cores}, $\Gamma x = s$ with a random $s \in \z2^m$ has the same satisfiability probability as $\fin \Gamma x = \fin s$ with a random $\fin s \in \z2^{\fin m}$.
    By \Cref{lem:maintain-uniformity-1}, $\fin \Gamma$ is uniformly distributed on the space of 2-cores, conditioned on its size $(\fin m, \fin\bn)$.
    This will allow applying \Cref{thm:main-theorem-core-sat} later in the proof. 
    Note 
    \begin{align*}
        \cThreshK > \cTurnK > 2.
    \end{align*}
    If $c > \cTurnK$, then 
    by \Cref{lem:bwz-core-reduction}, as $m,n\to\infty$ with $\lim m/n = c$, \aas{} we have $\fin m, \fin n_j \to \infty$, and
    \begin{align*}
        \lim_{m,n \to \infty} \frac{\fin m}{\fin n_j}
        = Q(\muK(c))\foralldisplay{j\in \{1,\ldots,\ourK\}}.
    \end{align*}
    Suppose $c > \cThreshK$.  
    Then by 
    \Cref{lem:Q-beats-K-after-thresh}, $Q(\muK(c)) > \ourK$, so \aas{} $\fin m, \fin n_j \to \infty$ and $\lim \fin m / \fin n_j > \ourK$ for each $j \in \{1,\ldots\ourK\}$. By \Cref{it:unsat-when-tall} of \Cref{thm:main-theorem-core-sat}, we conclude $\fin \Gamma x = \fin s$ is \aas{} unsatisfiable, with satisfiability probability $O(2^{-(\fin m - \abs{\fin \bn})})$. This proves \Cref{it:main-thm-aas-unsat} of \Cref{thm:main-thm-satisfiability}.

    Suppose $c \in (\cTurnK, \cThreshK)$. 
    Then $Q(\muK(c)) < \ourK$, so \aas{} $\fin m, \fin n_j \to\infty$ and $\lim \fin m / \fin n_j < \ourK$ for each $j \in \{1,\ldots\ourK\}$. By \Cref{it:sat-when-wide} of \Cref{thm:main-theorem-core-sat}, we conclude $\fin\Gamma x = \fin s$ is \aas{} satisfiable, with satisfiability probability $1 - O(\fin m^{2 - \ourK})$. 
    This proves \Cref{it:main-thm-aas-sat} of \Cref{thm:main-thm-satisfiability} for $c > \cTurnK$. 

    Suppose $c \in (2, \cTurnK]$. We use a monotonicity argument. Instead of considering games $(\Gamma,s)$ generated of size $(m,n)$ as $m',n\to\infty$ with $\lim m'/n = c$, consider games $(\Gamma',s')$ generated of size $(m',n)$ as $m,n\to\infty$ with $\lim m/n = c'$, where $c' > c$ and $c' \in (\cTurnK, \cThreshK)$. Note the uniform distribution of $(\Gamma,s)$ equations of size $(m,n)$ can be obtained by taking $(\Gamma',s')$ equations generated of size $(m',n)$ and truncating from $m'$ equations down to $m$ equations. By the previous paragraph, the games of size $(m',n)$ are \aas{} satisfiable as $m',n\to\infty$ with $m'/n \to c' \in (\cTurnK,\cThreshK)$, so the games of size $(m,n)$ are \aas{} satisfiable as well, since truncation cannot take a system from satisfiable to unsatisfiable.
\end{proof}

\clearpage
\printbibliography
\clearpage

\section{Appendix: Minor details--{For Web Only}}

\subsection{Proof of $\betaK$ upper bound}
\label{sec:proof-betaK-upper-bound}

\begin{lemma}
    If $\ourK \geq 3$, then $\betaK < 0.2$.
\end{lemma}
\begin{proof}
    Since $\ourK/(2(\ourK-1)) > 1/2$, we have
    \begin{align*}
        \betaK \leq \exp(1 -\left(\ln(\sqrt{\ourK-1})\right)\Big/\left(\frac{1}{2} - \frac{1}{\ourK}\right) )
         = e(\ourK-1)^{-\ourK/(\ourK-2)} \leq e(\ourK-1)^{-1}.
    \end{align*}
    For $\ourK\geq 15$, we have $\betaK \leq e(\ourK-1)^{-1} \leq 0.2$. The remaining values $\ourK\in \{3,4,\ldots,14\}$ can be checked one-by-one to satisfy $\betaK \leq 0.2$.
\end{proof}

\subsection{Proof of $\deltaBallK$ existence}

\begin{lemma}
    \label{lem:deltaBallK}
    There exists positive functions $\deltaBallK$ and $\epsBallK$ continuous on $c \in (2,\ourK)$ such that for all $c\in (2,\ourK)$, we have $\deltaBallK(c) < \betaK$ and
    \begin{align*}
        \JK(\alpha,(1-\deltaBallK(c),\deltaBallK(c));c) \leq -\epsBallK(c) \foralldisplay{\alpha \in [1-\deltaBallK(c), 1]}.
    \end{align*}
\end{lemma}
\begin{proof}
    This is Equation 5.30 of \cite{pittel2016satisfiability}. We elaborate a proof since they only mention it as a direct consequence of the continuity of $\JK$.

    Suppose $c>2$ and $\ourK \geq 3$. From direct substitution into the definition of $\JK$ in \Cref{def:Jkm}, we see
    \begin{align*}
        \JK(1,(1,0);c) &= \frac{1}{c} \ln \left(\frac{\expMTwo(\lambda) + \expMTwo(-\lambda)}{2 \expMTwo(\lambda)}\right).
    \end{align*}
    Note $\lambda = Q^{-1}(c) > 0$ because $c>2$. Hence $\expMTwo(\lambda) > \expMTwo(-\lambda)$, so
    \begin{align*}
        \JK(1,(1,0);c) < 0.
    \end{align*}

    Let $\epsBallK(c) = - \JK(1,(1,0);c)/2$, so $\epsBallK$ is positive and continuous in $c$.
    
    Fix $\ourK \geq 3$. For $n\geq 0$, define $I_n$ as
    \begin{align*}
        I_n \coloneqq (K - 1/n, \,K-1/(n+1)].
    \end{align*}
    For each $n\geq 1$, since $\JK$ is continuous, it is in particular uniformly continuous in $c \in \ov{I_n}$ since $\ov{I_n}$ is compact. Hence there exists some $\delta_n > 0$ such that 
    \begin{align*}
        \JK(\alpha,(\zeta_1,\zeta_2),c) \leq -\epsK(c) \foralldisplay{c \in \ov{I_n},\  \abs{(\alpha,\zeta_1,\zeta_2) - (1,1,0)} < \delta_n}.
    \end{align*}
    Define $\deltaBallK''$ on $c \in (K-1,K] = \bigcup_{n=1}^\infty I_n$ as constant on each $I_n$, such that for all $c \in I_n$,
    \begin{align*}
        \deltaBallK''(c) \coloneqq \min\{\betaK/2, \delta_{n}\}.
    \end{align*}
    This function $\deltaBallK''$ is piecewise constant and positive, so there exists a positive continuous piecewise-linear lower bound $\deltaBallK'(c) \leq \deltaBallK''(c)$ defined on $c \in [K-1,K)$.

    Hence we have positive continuous functions $\epsBallK$ and $\deltaBallK'$ such that
    \begin{align*}
        \JK(\alpha,(\zeta_1,\zeta_2);c) \leq -\epsBallK(c) \foralldisplay{\abs{(\alpha,\zeta_1,\zeta_2)-(1,1,0)} < \deltaBallK'(c)}.
    \end{align*}
    The result follows by defining the function $\deltaBallK \coloneqq \deltaBallK'/2$ since the segment $\{(\alpha,1-\delta,\delta)\mid \alpha\in [1-\delta,1]\}$ is a subset of the ball $\abs{(\alpha,\zeta_1,\zeta_2)-(1,1,0)} < 2\delta$.
\end{proof}

\subsection{Upper bound on reciprocal of binomial}

\begin{lemma}
    \label{lem:binomial-inv-upper-bound-appendix}
    There exists a constant $C$ such that for all $m\geq 1$ and $1\leq \ell \leq m$,
    \begin{align*}
        \binom{m}{\ell}^{-1} \leq C \sqrt{\ell} \exp(-mH(\alpha)),
    \end{align*}
    where $\alpha = \ell/m$ and $\mathdf{H(\alpha)}\coloneqq -\alpha \ln(\alpha) - (1-\alpha)\ln(1-\alpha)$ is the entropy function.
\end{lemma}

\begin{proof}
    Stirling's approximation states
    \begin{align*}
        \frac{\sqrt{2\pi n} e^{-n} \exp(n \ln(n))}{n!} \to 1
    \end{align*}
    as $n\to\infty$. This implies the existence of a constant $C_1$ such that
    \begin{align*}
        \frac{1}{C_1} \leq \frac{\sqrt{n} e^{-n} \exp(n \ln(n))}{n!} \leq C_1
    \end{align*}
    for all $n\geq 1$. Applying this bound to each involved factorial gives the existence of a constant $C$ such that
    \begin{align*}
        \binom{m}{\ell}^{-1} = \frac{\ell!(m-\ell)!}{m!} \leq C \sqrt{m \alpha (1-\alpha)} \exp(-mH(\alpha)).
    \end{align*}
    for all $\ell,m\geq 1$ with $m - \ell \geq 1$, where $\alpha = \ell/m$. Since $1-\alpha \leq 1$, we get
    \begin{align*}
        \binom{m}{\ell}^{-1} = \frac{\ell!(m-\ell)!}{m!} \leq C \sqrt{m \alpha} \exp(-mH(\alpha)).
    \end{align*}
   Note this bound holds in particular for $m = \ell\geq 1$, so we can loosen the $m-\ell\geq 1$ constraint to $m-\ell \geq 0$, proving the result.
\end{proof}

\newpage
\begin{center}
    {\LARGE NOT FOR PUBLICATION}
\end{center}

{
\label{table-of-contents}
\tableofcontents
}
\clearpage

\printindex

\end{document}